\documentclass[11pt, a4paper]{article}

\pdfoutput=1

\usepackage[utf8]{inputenc}
\usepackage[T1]{fontenc}
\usepackage[colorlinks]{hyperref}
\usepackage{microtype}

\hypersetup{
  linkcolor=[rgb]{0.3,0.3,0.6},
  citecolor=[rgb]{0.2, 0.6, 0.2},
  urlcolor=[rgb]{0.6, 0.2, 0.2}
}
\usepackage[capitalise, poorman]{cleveref}
\usepackage{mathtools, amsthm, amsfonts, amssymb}

\usepackage{mathrsfs}
\usepackage{pbox}

\usepackage{booktabs}
\usepackage{braket}
\usepackage{commath}
\usepackage{caption}
\usepackage{titlesec}
\titlelabel{\thetitle. }
\titleformat*{\section}{\large\bfseries}

\usepackage{accents}

\usepackage{tikz}
\usetikzlibrary{positioning}

\makeatletter
\newcommand{\raisemath}[1]{\mathpalette{\raisem@th{#1}}}
\newcommand{\raisem@th}[3]{\raisebox{#1}{$#2#3$}}
\makeatother

\theoremstyle{plain}
\newtheorem{theorem}{Theorem}[subsection]
\newtheorem{proposition}[theorem]{Proposition}
\newtheorem{lemma}[theorem]{Lemma}
\newtheorem{corollary}[theorem]{Corollary}

\theoremstyle{definition}
\newtheorem{definition}[theorem]{Definition}
\newtheorem{notation}[theorem]{Notation}

\newtheorem*{problem*}{Problem}
\newtheorem{problem}[theorem]{Problem}
\newtheorem{remark}[theorem]{Remark}
\newtheorem{example}[theorem]{Example}
\newtheorem{question}[theorem]{Question}

\newtheorem{innercustomtheorem}{Theorem}
\newenvironment{customtheorem}[1]
  {\renewcommand\theinnercustomtheorem{#1}\innercustomtheorem}
  {\endinnercustomtheorem}

\DeclarePairedDelimiter\ceil{\lceil}{\rceil}
\DeclarePairedDelimiter\floor{\lfloor}{\rfloor}

\DeclareMathOperator{\supp}{supp}

\DeclareMathOperator{\rank}{R}
\DeclareMathOperator{\subrank}{Q}
\DeclareMathOperator{\bordersubrank}{\underline{Q}}
\DeclareMathAccent{\wtilde}{\mathord}{largesymbols}{"65}

\DeclareMathOperator{\borderrank}{\underline{R}}

\DeclareMathOperator{\tens}{T}

\newcommand{\Oh}{\mathcal{O}}

\newcommand{\CC}{\mathbb{C}}
\newcommand{\QQ}{\mathbb{Q}}
\newcommand{\NN}{\mathbb{N}}
\newcommand{\GL}{\mathrm{GL}}

\newcommand{\ZZ}{\mathbb{Z}}
\newcommand{\RR}{\mathbb{R}}

\newcommand{\E}{\mathbb{E}}

\DeclareMathOperator{\type}{type}

\DeclareMathOperator{\Span}{Span}

\newcommand{\defin}[1]{\emph{#1}}

\newcommand{\CW}{\mathrm{CW}}

\newcommand{\optset}[1]{\mathscr{#1}}

\newcommand{\subomega}{\mathrm{q}}
\newcommand{\combsubomega}{\mathrm{q}_\mathrm{M}}
\newcommand{\combsubrank}{\subrank_\mathrm{M}}
\DeclareMathOperator{\bordercombsubrank}{\underline{\mathrm{Q}}_{\mathrm{M}}}
\newcommand{\combomega}{\omega_\mathrm{M}}
\newcommand{\badpairs}{\mathcal{C}}
\newcommand{\combgeq}{\geq_\mathrm{M}}
\newcommand{\combleq}{\leq_\mathrm{M}}
\newcommand{\combdegengeq}{\unrhd_\mathrm{M}}
\newcommand{\combdegenleq}{\unlhd_\mathrm{M}}
\newcommand{\degengeq}{\unrhd}
\newcommand{\degenleq}{\unlhd}

\newcommand{\gtens}[1]{\tens\Bigl(\hspace{0.1ex}#1\hspace{0.1ex}\Bigr)}
\newcommand{\gtensq}[1]{\tens_q\Bigl(\hspace{0.1ex}#1\hspace{0.1ex}\Bigr)}
\newcommand{\gtensf}[2]{\tens_{\!#2}\Bigl(\hspace{0.1ex}#1\hspace{0.1ex}\Bigr)}

\newcommand{\marg}[1]{\mathbf{#1}}

\begin{document}

\vspace*{1em}
\begin{center}
\Large\textbf{Asymptotic tensor rank of graph tensors:\\ beyond matrix multiplication}\par%
\vspace{1em}
\large Matthias Christandl\footnote{QMATH, Department of Mathematical Sciences, University of Copenhagen, Universitetsparken 5, 2100 Copenhagen, Denmark. Email: christandl@math.ku.dk}\!, Péter Vrana\footnote{Department of Geometry, Budapest University of Technology and Economics, Egry József~u.~1., 1111 Budapest, Hungary. Email: vranap@math.bme.hu} and Jeroen Zuiddam\footnote{QuSoft, CWI Amsterdam and University of Amsterdam, Science Park 123, 1098 XG Amsterdam, Netherlands. Email: j.zuiddam@cwi.nl}\par
\end{center}
\vspace{0.5em}

\begin{abstract}
We present an upper bound on the exponent of the asymptotic behaviour of the tensor rank
of a family of tensors defined by the complete graph on $k$ vertices. For~$k\geq4$, we show that the exponent per edge is at most 0.77, %
outperforming the best known upper bound on the exponent per edge for matrix multiplication ($k=3$), which is approximately~0.79.
We raise the question whether for some $k$ the exponent per edge can be below $2/3$, i.e.\
can outperform matrix multiplication even if the matrix multiplication exponent equals~2.
In order to obtain our results, we generalise to higher order tensors a result by Strassen on the asymptotic subrank of tight tensors and a result by Coppersmith and Winograd on the asymptotic rank of matrix multiplication. Our results have applications in entanglement theory and communication complexity.
\end{abstract}

\section{Introduction}

A famous open problem in algebraic complexity theory is the problem of determining the tensor rank of large tensor powers of the $2\times 2$ matrix multiplication tensor. Let $V_1, \ldots, V_k$ be finite-dimensional complex vector spaces and let $\phi \in V_1 \otimes \cdots \otimes V_k$ be a $k$-tensor. The tensor rank $\rank(\phi)$ of $\phi$ is the smallest number~$r$ such that $\phi$ can be written as a sum of $r$ simple tensors $v_1 \otimes \cdots \otimes v_k \in V_1 \otimes \cdots \otimes V_k$. We define the $N$th tensor power $\phi^{\otimes N}$ as the $k$-tensor obtained by taking the tensor product of $N$ copies of $\phi$ and grouping such that $\phi \in (V_1^{\otimes N}) \otimes \cdots \otimes (V_k^{\otimes N})$.
For any $n$, let $b_1, \ldots, b_n$ denote the standard basis of $\CC^n$.
The $2\times 2$ matrix multiplication tensor is the 3-tensor
\[
\sum_{\mathclap{i\in \{0,1\}^3}} (b_{i_1} \otimes b_{i_2}) \otimes  (b_{i_2} \otimes b_{i_3}) \otimes  (b_{i_3} \otimes b_{i_1}) \in (\CC^2 \otimes \CC^2)^{\otimes 3}.
\]
This paper is motivated by the study of the tensor rank of powers of tensors that are generalizations of the matrix multiplication tensor.  %
Consider a graph with two vertices connected by a single edge. We define the corresponding 2-tensor as
\def\grsizex{1.2em}
\def\grsize{1.2em}
\[
\tens\bigl(\pbox{4em}{
\begin{tikzpicture}[line width=0.2mm, vertex/.style={
    circle,
    fill=white,
    draw,
    outer sep=0pt,
    inner sep=0.1em}, x=\grsizex, y=\grsizex]
    \path[coordinate] (0,0)  coordinate(A) (1,0) coordinate(B);
	\draw [line width=0.2mm]  (A) node [vertex] {} -- (B) node [vertex] {};
\end{tikzpicture}}\bigr) = \sum_{\mathclap{i\in\{0,1\}}} b_i \otimes b_i\, \in \CC^2 \otimes \CC^2.
\]
Let $G$ be any graph. Then we define $\tens(G)$ to be the $|V|$-tensor obtained by taking the tensor product of the tensors corresponding to the edges of~$G$, grouping corresponding vertices together (the full definition is in \cref{tensdef}). For example, for the complete graph on four vertices $K_4$ we have
\newcommand{\smallotimes}{\hspace{-0.3em}\otimes\!}
\begin{align*}
\gtens{\pbox{2em}{
\begin{tikzpicture}[line width=0.2mm, vertex/.style={
    circle,
    fill=white,
    draw,
    outer sep=0pt,
    inner sep=0.1em}, x=\grsize, y=\grsize]
    \path[coordinate] (0,0)  coordinate(A)
                ( 0,1) coordinate(B)
                ( 1,1) coordinate(C)
                ( 1,0) coordinate(D);
	\draw [line width=0.2mm] (A) -- (C);
	\draw [line width=0.2mm] (B) -- (D);
	\draw [line width=0.2mm]  (A) node [vertex] {} -- (B) node [vertex] {} -- (C) node [vertex] {} -- (D) node [vertex] {} -- (A);
\end{tikzpicture}
}} &=  \gtens{\pbox{1.9em}{
\begin{tikzpicture}[line width=0.2mm, vertex/.style={
    circle,
    fill=white,
    draw,
    outer sep=0pt,
    inner sep=0.1em}, x=\grsize, y=\grsize]
    \path[coordinate] (0,0)  coordinate(A)
                ( 0,1) coordinate(B)
                ( 1,1) coordinate(C)
                ( 1,0) coordinate(D);
	\draw (A) node [vertex] {};
	\draw (B) node [vertex] {};
	\draw (C) node [vertex] {};
	\draw (D) node [vertex] {};
	\draw [line width=0.2mm]  (A) node [vertex] {} -- (B) node [vertex] {};
\end{tikzpicture}
}}
\otimes
\gtens{\pbox{1.9em}{
\begin{tikzpicture}[line width=0.2mm, vertex/.style={
    circle,
    fill=white,
    draw,
    outer sep=0pt,
    inner sep=0.1em}, x=\grsize, y=\grsize]
    \path[coordinate] (0,0)  coordinate(A)
                ( 0,1) coordinate(B)
                ( 1,1) coordinate(C)
                ( 1,0) coordinate(D);
	\draw (A) node [vertex] {};
	\draw (B) node [vertex] {};
	\draw (C) node [vertex] {};
	\draw (D) node [vertex] {};
	\draw [line width=0.2mm]  (A) node [vertex] {} -- (C) node [vertex] {};
\end{tikzpicture}
}}
\otimes
\gtens{\pbox{1.9em}{
\begin{tikzpicture}[line width=0.2mm, vertex/.style={
    circle,
    fill=white,
    draw,
    outer sep=0pt,
    inner sep=0.1em}, x=\grsize, y=\grsize]
    \path[coordinate] (0,0)  coordinate(A)
                ( 0,1) coordinate(B)
                ( 1,1) coordinate(C)
                ( 1,0) coordinate(D);
	\draw (A) node [vertex] {};
	\draw (B) node [vertex] {};
	\draw (C) node [vertex] {};
	\draw (D) node [vertex] {};
	\draw [line width=0.2mm]  (A) node [vertex] {} -- (D) node [vertex] {};
\end{tikzpicture}
}} \otimes
\gtens{\pbox{1.9em}{
\begin{tikzpicture}[line width=0.2mm, vertex/.style={
    circle,
    fill=white,
    draw,
    outer sep=0pt,
    inner sep=0.1em}, x=\grsize, y=\grsize]
    \path[coordinate] (0,0)  coordinate(A)
                ( 0,1) coordinate(B)
                ( 1,1) coordinate(C)
                ( 1,0) coordinate(D);
	\draw (A) node [vertex] {};
	\draw (B) node [vertex] {};
	\draw (C) node [vertex] {};
	\draw (D) node [vertex] {};
	\draw [line width=0.2mm]  (B) node [vertex] {} -- (C) node [vertex] {};
\end{tikzpicture}
}}
\otimes
\gtens{\pbox{1.9em}{
\begin{tikzpicture}[line width=0.2mm, vertex/.style={
    circle,
    fill=white,
    draw,
    outer sep=0pt,
    inner sep=0.1em}, x=\grsize, y=\grsize]
    \path[coordinate] (0,0)  coordinate(A)
                ( 0,1) coordinate(B)
                ( 1,1) coordinate(C)
                ( 1,0) coordinate(D);
	\draw (A) node [vertex] {};
	\draw (B) node [vertex] {};
	\draw (C) node [vertex] {};
	\draw (D) node [vertex] {};
	\draw [line width=0.2mm]  (B) node [vertex] {} -- (D) node [vertex] {};
\end{tikzpicture}
}}
\otimes
\gtens{\pbox{1.9em}{
\begin{tikzpicture}[line width=0.2mm, vertex/.style={
    circle,
    fill=white,
    draw,
    outer sep=0pt,
    inner sep=0.1em}, x=\grsize, y=\grsize]
    \path[coordinate] (0,0)  coordinate(A)
                ( 0,1) coordinate(B)
                ( 1,1) coordinate(C)
                ( 1,0) coordinate(D);
	\draw (A) node [vertex] {};
	\draw (B) node [vertex] {};
	\draw (C) node [vertex] {};
	\draw (D) node [vertex] {};
	\draw [line width=0.2mm]  (C) node [vertex] {} -- (D) node [vertex] {};
\end{tikzpicture}
}}\\
&= \sum_{\mathclap{i\in\{0,1\}^6}} (b_{i_1}\smallotimes b_{i_4}\smallotimes b_{i_5}) \otimes (b_{i_2}\smallotimes b_{i_4}\smallotimes b_{i_6}) \otimes (b_{i_3}\smallotimes b_{i_5}\smallotimes b_{i_6}) \otimes (b_{i_1}\smallotimes b_{i_2}\smallotimes b_{i_3})
\end{align*}
living in $(\CC^8)^{\otimes 4}$. We can ignore the dependence of this tensor on the order of the edges, since tensor rank is invariant under this choice. Let $K_k$ be the complete graph on $k$ vertices. The $2\times 2$ matrix multiplication tensor is the tensor $\tens(K_3)$.

The main result of this paper is an upper bound on the tensor rank of large tensor powers of $\tens(K_k)$: for any $k\geq 4$ and large $N$,
\begin{equation}\label{maineq}
\rank(\tens(K_k)^{\otimes N}) \leq 2^{0.772943 \binom{k}{2} N + o(N)}.
\end{equation}
We say that the \emph{exponent} $\omega(\tens(K_k))$ is at most $0.772943 \binom{k}{2}$ and we say that the \emph{exponent per edge} $\tau(\tens(K_k))$ is at most $0.772943$.
This improves, for $k\geq 4$, the bound $\tau(\tens(K_k)) \leq 0.790955$ that can be derived from the well-known upper bound of Le~Gall \cite{le2014powers} on the exponent of matrix multiplication $\omega\coloneqq\omega(\tens(K_3))$. Note that $\tau(\tens(K_3)) = \omega/3$.

By a ``covering argument'' we can show that $\tau(\tens(K_k))$ is nonincreasing when~$k$ increases (\cref{peredge}).
On the other hand, a standard ``flattening argument'' (see \cref{flattening}) yields the lower bound $\tau(\tens(K_k)) \geq \tfrac12k/{(k-1)}$ if $k$ is even and $\tau(\tens(K_k)) \geq \tfrac12(k+1)/k$ if $k$ is odd. As a consequence, if the exponent of matrix multiplication $\omega$ equals~$2$, then $\tau(\tens(K_4)) = \tau(\tens(K_3)) = \tfrac23$. We raise the following question: is there a $k\geq 5$ such that $\tau(\tens(K_k)) < \tfrac23$? More open questions are discussed in \cref{subsecquestions}.

Our method to prove \eqref{maineq} is a generalization of a method of Strassen, and Coppersmith and Winograd for obtaining upper bounds on the exponent of the matrix multiplication tensor.
To use the generalized Coppersmith--Winograd method we have to get a handle on the monomial subrank of tensor powers of yet another type of combinatorially defined tensors. %
The definition of rank given above is equivalent to saying that the rank $\rank(\phi)$ of a tensor $\phi \in V_1\otimes \cdots \otimes V_k$ is the smallest number $r$ for which there exist linear maps $A_1: \CC^r \to V_1$, \ldots, $A_k : \CC^r \to V_k$ such that $\phi = (A_1\otimes \cdots \otimes A_k) \tens_r(k)$, where $\tens_r(k) = \sum_{i=1}^r (b_i)^{\otimes k}$ is the rank-$r$ unit $k$-tensor.
The subrank $Q(\phi)$ of~$\phi$ is the largest number $s$ for which there exist linear maps $A_1: V_1 \to \CC^s$, \ldots, $A_k: V_k \to \CC^s$ such that $\tens_s(k) = (A_1\otimes \cdots \otimes A_k)\phi$. In fact, we need a slightly restricted notion. A monomial matrix is a matrix such that any row or column has at most one nonzero entry. %
Let $\phi\in V_1 \otimes \cdots \otimes V_k$ be a tensor in a fixed basis. The \emph{monomial} subrank $\combsubrank(\phi)$ is the largest number $s$ for which there exist monomial matrices $A_1, \ldots, A_k$ such that $\tens_s(k) = (A_1\otimes \cdots \otimes A_k) \phi$. We are interested in lower bounding $\combsubrank(\phi^{\otimes N})$ for large $N$.
In particular, in the course of proving \eqref{maineq} we face the problem of computing $\combsubrank(D_{(2,2)}^{\otimes N})$, where %
\begin{alignat*}{4}
D_{(2,2)} ={}&\,  b_1\otimes b_1\otimes b_2\otimes b_2
 {}+{}&\, b_1\otimes b_2\otimes b_1\otimes b_2
 {}+{}&\, b_1\otimes b_2\otimes b_2\otimes b_1\\
 {}+{}&\, b_2\otimes b_1\otimes b_1\otimes b_2
 {}+{}&\, b_2\otimes b_1\otimes b_2\otimes b_1
 {}+{}&\, b_2\otimes b_2\otimes b_1\otimes b_1
\end{alignat*}
is the weight-$(2,2)$ Dicke tensor.

Our second result solves this problem. Namely, we prove 
a general asymptotic lower bound on the monomial subrank of tensor powers of so-called tight tensors (\cref{main}). These are tensors $\phi \in V_1 \otimes \cdots \otimes V_k$ for which there is a choice of bases $B_1, \ldots, B_k$ for $V_1, \ldots, V_k$ respectively and injective maps $\alpha_1:B_1\to \ZZ$, \ldots, $\alpha_k : B_k \to \ZZ$ such that
\[
\forall\, (b_1, \ldots, b_k)\in \supp_B(\phi)\quad \alpha_1(b_1) + \cdots + \alpha_k(b_k) = 0.
\]
For example, applied to the tensor $D_{(2,2)}$, which is a tight tensor, our result yields the  monomial subrank $\combsubrank(D_{(2,2)}^{\otimes N}) = 2^{N - o(N)}$, which is asymptotically optimal. We say that the \emph{monomial subexponent} $\combsubomega(D_{(2,2)})$ equals 1. This solves a conjecture posed in \cite{vrana2015asymptotic} for the special case $D_{(2,2)}$.

In quantum information theoretical terms, the tensor $\tens(K_k)$ encodes a $k$-partite quantum state in which each two systems share an Einstein--Podolski--Rosen (EPR) pair. The tensor $\tens_2(k)$ encodes the $k$-partite Greenberger--Horne--Zeilinger (GHZ) state. If the rank of $\tens(K_k)$ is $r$, then $\ceil{\log_2 r}$ copies of the state $\tens_2(k)$ are needed to generate $\tens(K_k)$ by stochastic local operations and classical communication (SLOCC). See \cite{dur2000three} for a discussion of SLOCC. Asymptotically, %
we can generate $N$ copies of $\tens(K_k)$ using $\omega(\tens(K_k))N + o(N)$ copies of $\tens(k)$.

In the next subsection we discuss preliminary definitions and known results. After that we state our  results and discuss open questions.

\subsection{Preliminaries}

Let $U_1,\ldots, U_k$ and $V_1, \ldots, V_k$ be complex finite-dimensional vector spaces. For any number $n\in \NN$, define the set $[n] \coloneqq \{1,2,\ldots, n\}$. %
All our graphs will be simple graphs, that is, they are unweighted, undirected, containing no self-loops or multiple edges.

We first define the two families of tensors that play an important role in this paper, namely the tensors $\tens(G)$ with $G$ a graph and the tensors $D_\lambda$ with $\lambda$ a partition. We start with $\tens(G)$.

\begin{definition}\label{tensdef}
Let $G = (V,E)$ be a graph and let $n$ be a natural number. Let $b_1, \ldots, b_n$ be the standard basis of $\CC^n$. We define the $|V|$-tensor $\tens_n(G)$ as
\[
\tens_n(G) \coloneqq \!\!\sum_{i\in [n]^E} \bigotimes_{v\in V} \Bigl( \bigotimes_{\substack{e\in E:\\ v\in e}} b_{i_e} \Bigr),
\]
where the sum is over all tuples $i$ indexed by $E$ with entries in $[n]$. %
Equivalently, one can define $\tens_n(G)$ as a tensor product over edges of $G$, as follows, letting subscripts denote the position of tensor legs,
\[
\tens_n(G) = \bigotimes_{e \in E} \sum_{i\in [n]}  (b_i \otimes b_i)_e \otimes (1\otimes \cdots \otimes 1)_{V\setminus e}.
\]
Here, the element $1$ should be regarded as an element of $\CC$ and the large tensor product inputs and outputs $|V|$-tensors by the natural regrouping. We will denote $\tens_2(G)$ by $\tens(G)$. The definition naturally generalizes to hypergraphs, see \cite{vrana}.
\end{definition}

Note that the tensor $\tens_n(G)$ behaves as follows under the tensor Kronecker product: $\tens_n(G) \otimes T_m(G) \cong T_{nm}(G)$.

We can ignore the fact that the tensor in the above definition depends on the choice of order of the edges and vertices of the graph $G$, since tensor rank and subrank do not depend on this order.

Let $K_k$ be the complete graph on~$k$ vertices and let $C_k$ be the cycle graph on $k$ vertices.

\begin{example} We give some examples of tensors of type $\tens_n(G)$. Let~$n\in\NN$. For any $j_1,j_2,j_3\in [n]$, we define $b_{j_1j_2} \coloneqq b_{j_1} \otimes b_{j_2}$ and $b_{j_1j_2j_3} \coloneqq b_{j_1} \otimes b_{j_2} \otimes b_{j_3}$. Then
\begin{align*}
\tens_n(K_2) &= \sum_{\mathclap{i\in [n]^2}} b_{i_1} \otimes b_{i_1} \in \CC^n \otimes \CC^n \\
\tens_n(K_3) &= \sum_{\mathclap{i\in [n]^3}} b_{i_1 i_2} \otimes b_{i_2 i_3} \otimes b_{i_3 i_1} \in \CC^{n^2}\! \otimes \CC^{n^2}\! \otimes \CC^{n^2}\\
\tens_n(K_4) &= \sum_{\mathclap{i\in[n]^6}} b_{i_1 i_2 i_5} \otimes b_{i_2 i_3 i_6} \otimes b_{i_3 i_4 i_5} \otimes b_{i_1 i_4 i_6}  \in (\CC^{n^3})^{\otimes 4}\\
\tens_n(C_5) &= \sum_{\mathclap{i\in[n]^5}} b_{i_1 i_2} \otimes b_{i_2 i_3} \otimes b_{i_3 i_4} \otimes b_{i_4 i_5} \otimes b_{i_5 i_1}  \in (\CC^{n^2})^{\otimes 5}
\end{align*}
In algebraic complexity theory, the tensor $\tens_n(C_3)$ is called the $n\times n$ matrix multiplication tensor, since it encodes the bilinear map that multiplies two $n\times n$ matrices. It is usually denoted by $\langle n,n,n \rangle$. For $k\geq 3$, $\tens_n(C_k)$ is the iterated matrix multiplication tensor; it encodes the multilinear map that multiplies $k$ matrices of size $n\times n$.
Note that $C_3 = K_3$, so any results on $\tens_n(K_k)$ or $\tens_n(C_k)$ are generalizations of results on the matrix multiplication tensor.
\end{example}

We now introduce the second family of tensors, $D_\lambda$.

\begin{definition}[Dicke tensor \cite{dicke1954coherence, stockton2003characterizing, vrana}]
Let $k$ be a positive integer and let $\lambda = (\lambda_1,\ldots,\lambda_n)  \vdash_n k$ be a partition of $k$ of length at most~$n$. Let $b_1, \ldots, b_n$ be the standard basis of $\CC^n$. Define the \defin{weight-$\lambda$ Dicke tensor}~$D_\lambda$ as
\[
D_\lambda \coloneqq \sum_{\mathclap{\substack{i \in [n]^k:\\ \type(i)=\lambda}}} b_{i_1} \otimes \cdots \otimes b_{i_k} \in (\CC^n)^{\otimes k}
\]
where $\type(i)=\lambda$ means that $i$ is a permutation of the tuple
\[
(\underbrace{1,\ldots, 1}_{\lambda_1}, \underbrace{2, \ldots, 2}_{\lambda_2}, \ldots, \underbrace{n, \ldots, n}_{\lambda_n}).
\]
\end{definition}

\begin{example} We give some examples of tensors of type $D_\lambda$.
\begin{alignat*}{2}
D_{(2,1)} {}={}&\, b_1 \otimes b_1 \otimes b_2 + b_1 \otimes b_2 \otimes b_1 + b_2 \otimes b_1 \otimes b_1,\\[0.5em]
D_{(1,1,1)} {}={}&\, b_1 \otimes b_2 \otimes b_3 + b_1 \otimes b_3 \otimes b_2 + b_2 \otimes b_1 \otimes b_3\\ {}+{}&\, b_2 \otimes b_3 \otimes b_1 + b_3 \otimes b_1 \otimes b_2 + b_3 \otimes b_2 \otimes b_1.
\end{alignat*}
In quantum information theory, the tensor $D_{(2,1)}$ encodes a quantum state known as the W-state, and $D_{(k-1,1)}$ encodes a generalized W-state on $k$ systems.
\end{example}

We now wish to define the exponent and (monomial) subexponent of a tensor, which should be thought of as our complexity measures for tensors. First we define (monomial) restriction, degeneration and asymptotic conversion rate.

Let $\phi \in U_1\otimes \cdots \otimes U_k$ and $\psi\in V_1 \otimes \cdots \otimes V_k$ be $k$-tensors.

\begin{definition}\label{restrict}
We say $\phi$ \defin{restricts to} $\psi$, written $\phi \geq \psi$,  if there exist linear maps $A_i: U_i\to V_i$ such that $\psi = (A_1\otimes \cdots \otimes A_k) \phi$. We say $\phi$ is \defin{isomorphic} to $\psi$ if both $\phi\geq \psi$ and $\phi \leq \psi$. We will often tacitly treat isomorphic tensors as being equal. We say $\phi$ \defin{degenerates to} $\psi$, written $\phi \degengeq \psi$, if $\psi$ is in the orbit closure $\overline{G \phi}$ in the Zariski topology, where $G = \GL(U_1) \times \cdots \times \GL(U_k)$.

A matrix is a monomial matrix if on any of its rows or columns there is at most one nonzero entry.
Fix bases for $U_1, \ldots, U_k$ and $V_1, \ldots, V_k$. We say that there is a \defin{monomial restriction} from $\phi$ to $\psi$, written $\phi \combgeq \psi$, if there exist monomial matrices $A_1, \ldots, A_k$ such that $\psi = (A_1 \otimes \cdots \otimes A_k) \phi$ in the chosen basis. We say that there is a \defin{monomial degeneration} from $\phi$ to~$\psi$, written $\phi \combdegengeq \psi$, if $\psi$ is in the orbit closure $\overline{M \phi}$ in the Zariski topology, where~$M$ is the subgroup of $G$ consisting of $k$-tuples of monomial matrices. %
\end{definition}

Monomial degeneration has a nice combinatorial description for which we refer to Theorem~6.1 in  \cite{strassen1987relative}. It is clear from the definition that $\phi \combgeq \psi$ implies $\phi \combdegengeq \psi$. We refer to \cite{strassen1987relative} for other basic properties of $\combgeq$ and $\combdegengeq$.

\begin{definition}
Define the \defin{asymptotic conversion rate} from $\phi$ to $\psi$ as
\begin{equation}\label{omegadef}
\omega(\phi, \psi) \coloneqq \lim_{n\to\infty} \frac{1}{n} \min\{m \in \NN \mid \phi^{\otimes m} \geq \psi^{\otimes n} \},
\end{equation}
and the \defin{asymptotic monomial conversion rate} from $\phi$ to $\psi$ as
\begin{equation}\label{combomegadef}
\combomega(\phi, \psi) \coloneqq \lim_{n\to\infty} \frac{1}{n} \min\{m \in \NN \mid \phi^{\otimes m} \combdegengeq \psi^{\otimes n} \}.
\end{equation}
The minimum of the empty set is considered to be $\infty$.
\end{definition}

\begin{proposition}
The limit in Equation~\eqref{omegadef} exists and equals the supremum 
$\sup_n  \frac{1}{n} \min\{m \in \NN \mid \phi^{\otimes m} \geq \psi^{\otimes n} \}$.
The limit in Equation~\eqref{combomegadef} exists and equals the supremum $\sup_n  \frac{1}{n} \min\{m \in \NN \mid \phi^{\otimes m} \combdegengeq \psi^{\otimes n} \}$.
\end{proposition}
\begin{proof}
See Lemma 1.1 in \cite{strassen1988asymptotic}.
\end{proof}

\begin{theorem}\label{restrdegen}
Restriction and degeneration are asymptotically equivalent, in the sense that
\[
\omega(\phi, \psi) = \lim_{n\to\infty} \frac{1}{n} \min \{m\in \NN\mid \phi^{\otimes m} \degengeq \psi^{\otimes n} \}.
\]
\end{theorem}
\begin{proof}
See \cite{strassen1987relative}. A proof can also be found in \cite{vrana2015asymptotic}.
\end{proof}

\begin{remark}
We do not know whether in \eqref{combomegadef} one may equivalently replace~$\combdegengeq$ by $\combgeq$. We therefore defined $\combomega$ using the more powerful $\combdegengeq$.
\end{remark}

Strassen introduced the asymptotic study of restriction in the context of the algebraic complexity of bilinear maps \cite{strassen1988asymptotic, strassen1991degeneration}. In quantum information theory, if $\phi$ and $\psi$ are pure quantum states, then $\phi\geq \psi$ means precisely that $\psi$ can be obtained from $\phi$ by means of stochastic local operations and classical communication (SLOCC).
We refer to \cite{vrana2015asymptotic} and \cite{vrana} for general properties of $\omega(\phi, \psi)$.

Restriction and asymptotic conversion rate can be used to compare any two $k$-tensors. In this context, the following tensor will serve as an absolute reference tensor.

\begin{definition}[GHZ tensor or unit tensor]
Let $k,r\in \NN$, and let $b_1, \ldots, b_r$ be the standard basis of $\CC^r$. Define $\tens_r(k)$ as the $k$-tensor 
\[
\tens_r(k) \coloneqq \sum_{\smash{i\in[r]}} b_i\otimes \cdots \otimes b_i   \in  (\CC^r)^{\otimes k}.
\]
We will denote $\tens_r(k)$ by $\tens_r$ when the order~$k$ is understood and we will denote $\tens_2(k)$ by just $\tens(k)$ or $\tens$. %
In quantum information theory, the tensor $\tens_r(k)$ encodes the Greenberger--Horne--Zeilinger (GHZ) quantum state of rank~$r$ on~$k$ systems.
In algebraic complexity theory, for $k=3$ this tensor is called the rank-$r$ unit tensor and is denoted by~$\langle r \rangle$.
Note that $\tens_r(k)$ equals the hypergraph tensor $\tens_r(H)$ where $H$ is the hypergraph on $k$ vertices with a single hyperedge containing all vertices. We refer to \cite{vrana} for the definition of hypergraph tensor.
\end{definition}

\begin{definition}
The \defin{rank} of $\psi$, denoted $\rank(\psi)$, is the smallest number $r$ such that $\psi \leq \tens_r$. Equivalently, it is the smallest number $r$ such that $\psi$ can be written as a sum of $r$ simple tensors $v_1 \otimes \cdots \otimes v_k \in V_1 \otimes \cdots \otimes V_k$. The \defin{border rank} of $\psi$, denoted $\borderrank(\psi)$ is the smallest number $r$ such that $\psi \degenleq \tens_r$.
\end{definition}

\begin{definition}
The \defin{subrank} of $\phi$, denoted $\subrank(\phi)$, is the largest number~$s$ such that $\tens_s \leq \phi$. 
The \defin{monomial subrank} of $\phi$, denoted $\combsubrank(\phi)$, is the largest number $s$ such that $\tens_s \combleq \phi$.
In the same way, one defines the \defin{border subrank}~$\bordersubrank(\phi)$ and the  \defin{monomial border subrank}~$\bordercombsubrank(\phi)$ using $\degenleq$ and $\combdegenleq$ respectively.
\end{definition}

\begin{definition}
The %
 \defin{exponent} of $\psi$ is defined as the asymptotic conversion rate from $\tens$ to $\psi$ and is denoted by~$\omega(\psi)$,
\[
\omega(\psi) \coloneqq \omega(\tens, \psi).
\]
\end{definition}

\begin{definition}
The %
 \defin{subexponent} of $\phi$ is defined as the inverse of the asymptotic conversion rate from $\psi$ to $\tens$, and is denoted by $\subomega(\phi)$,
\[
\subomega(\phi) \coloneqq \omega(\phi, \tens)^{-1}.
\]
\end{definition}

\begin{definition}\label{defcombsubexp}
Let $\phi \in U_1 \otimes \cdots \otimes U_k$ be a $k$-tensor in a fixed basis. The %
 \defin{monomial subexponent} of $\phi$ is defined as the inverse of the asymptotic monomial conversion rate from $\psi$ to $\tens$, and is denoted by $\combsubomega(\phi)$,
\[
\combsubomega(\phi) \coloneqq \combomega(\phi, \tens)^{-1}.
\]
\end{definition}

The parameters $\omega$, $\subomega$ and $\combsubomega$ have the following two useful descriptions in terms of $\rank$, $\subrank$ and~$\bordercombsubrank$.

\begin{proposition}\label{limprop} Let $\phi$ and $\psi$ be $k$-tensors. Then,
\begin{alignat*}{4}
\omega(\psi) %
&= \lim_{n\to \infty} \tfrac1n \log_2 \rank(\psi^{\otimes n}),\\
\subomega(\phi) %
&= \lim_{m\to \infty} \tfrac1m \log_2 \subrank(\phi^{\otimes m}),\\
\combsubomega(\phi) %
&= \lim_{m\to \infty} \tfrac1m \log_2 \bordercombsubrank(\phi^{\otimes m}).
\end{alignat*}
\end{proposition}
\begin{proof}
To prove the first equality, we have
\begin{align*}
\omega(\tens, \psi) &= \lim_{n\to\infty} \frac1n \min\{m \in \NN\mid \tens^{\otimes m} \geq \psi^{\otimes n}\}\\
&= \lim_{n\to\infty} \frac1n \lceil \log_2 \rank(\psi^{\otimes n}) \rceil \\
&= \lim_{n\to\infty} \frac1n \log_2 \rank(\psi^{\otimes n}).
\end{align*}
To prove the second equality, we have
\begin{align*}
\omega(\phi, \tens) &= \lim_{n\to\infty} \frac1n \min\{m \in \NN \mid \phi^{\otimes m} \geq \tens^{\otimes n}\}\\
&= \lim_{m\to\infty} \max\{n\in \NN\mid \phi^{\otimes m} \geq \tens^{\otimes n}\}^{-1}\, m\\
&= \lim_{m\to\infty} \lfloor \log_2 \subrank(\phi^{\otimes m}) \rfloor^{-1}\, m\\
&= \lim_{m\to\infty} \bigl(\log_2 \subrank(\phi^{\otimes M})\bigr)^{-1} m\\
&= \subomega(\phi)^{-1}.
\end{align*}
The statement for $\combsubomega$ follows from a similar proof.
\end{proof}

\begin{proposition}\label{charac2} Let $\phi$ and $\psi$ be $k$-tensors. Then,
\begin{alignat*}{4}
\omega(\psi) %
&= \inf\{\beta \in \RR\mid \rank(\psi^{\otimes N}) \leq 2^{\beta N + o(N)}\}, \\
\subomega(\phi) %
&= \sup\{\beta \in \RR\mid \subrank(\phi^{\otimes N}) \geq 2^{\beta N - o(N)}\},\\
\combsubomega(\phi) %
&= \sup\{\beta \in \RR\mid \bordercombsubrank(\phi^{\otimes N}) \geq 2^{\beta N - o(N)}\}.
\end{alignat*}
\end{proposition}
\begin{proof}
If $\rank(\psi^{\otimes N}) \leq 2^{\beta N + o(N)}$, then $\omega(\psi) \leq \lim_{N\to \infty} \beta + o(N)/N = \beta$, by \cref{limprop}. Conversely, suppose that $\omega(\psi) < \beta$. Then $\rank(\psi^{\otimes N}) \leq \Oh(2^{\beta N})$, by \cref{limprop}, so $\rank(\psi^{\otimes N})\leq 2^{\beta N + o(N)}$. The statement for $\subomega$ and $\combsubomega$ follows similarly from \cref{limprop}.
\end{proof}

For tensors of type $\tens(G)$ we have the following characterization of $\omega$, $\subomega$ and $\combsubomega$.

\begin{proposition} Let $G$ be a graph. Then,
\begin{align*}
\omega(\tens(G)) &= \inf \{\beta \in \RR \mid \rank(\tens_n(G)) = \Oh(n^\beta)\},\\
\subomega(\tens(G)) &= \sup \{\beta \in \RR \mid \subrank(\tens_n(G)) = \Omega(n^\beta)\},\\
\combsubomega(\tens(G)) &= \sup \{\beta \in \RR \mid \bordercombsubrank(\tens_n(G)) = \Omega(n^\beta)\}.
\end{align*}
\end{proposition}
\begin{proof}
Suppose that $\rank(\tens_n(G)) = \Oh(n^{\beta})$. Then,
\[
\rank(\tens_2(G)^{\otimes N}) = \rank(\tens_{2^N}(G)) = \Oh(2^{N\beta}).
\]
so $\omega(\tens(G)) \leq \beta$ by \cref{limprop}. Here we used $\tens_2(G)^{\otimes N} \cong \tens_{2^N}(G)$. On the other hand, suppose that $\omega(\tens(G)) < \beta$. Then,
\[
\rank(\tens_{2^N}(G)) = \rank(\tens_2(G)^{\otimes N}) = \Oh(2^{\beta N})
\]
by \cref{limprop},
so
\[
\rank(\tens_n(G)) \leq \rank(\tens_{2^N}(G)) = \Oh(2^{\beta N}) = \Oh(n^\beta)
\]
where $N = \lceil \log_2 n \rceil$.

The proofs for $\subomega$ and $\combsubomega$ follow similarly from \cref{limprop}.
\end{proof}

Note that $\omega(\tens_r(G)) = \omega(\tens(G)) \log_2 r$.

\begin{definition}\label{flattening}
Let $\phi \in U_1 \otimes \cdots \otimes U_k$ be a $k$-tensor. A \defin{flattening} of $\phi$ is a 2-tensor (a matrix) obtained by any grouping of the tensor legs $U_1, \ldots, U_k$ into two groups. 
\end{definition}

Flattenings are useful for obtaining bounds on tensor rank and related notions. 
\begin{proposition}\label{flat}
Let $\phi$ be a $k$-tensor and let $A_\phi$ be any flattening of~$\phi$. Then $\rank(A_\phi) \leq \rank(\phi)$ and $\subrank(\phi) \leq \subrank(A_\phi)$. On the asymptotic level, we have the inequalities $\subomega(\phi) \leq \subomega(A_\phi) = \log_2 \subrank(A_\phi) = \log_2 \rank(A_\phi) = \omega(A_\phi) \leq \omega(\phi)$.
\end{proposition}
\begin{proof}
A flattening of a simple $k$-tensor is a simple 2-tensor (a rank-1 matrix), and hence for an arbitrary $k$-tensor $\phi$ we have $\rank(A_\phi) \leq \rank(\phi)$.
Similarly, $\subrank(\phi) \leq \subrank(A_\phi)$.

For matrices, rank is multiplicative under tensor product and coincides with subrank. Therefore, for any $n\in \NN$,
\[
\subrank(A^{\otimes n}) \leq \subrank(A_\phi^{\otimes n}) = \subrank(A_\phi)^n = \rank(A_\phi)^n = \rank(A_\phi^{\otimes n}) \leq \rank(\phi^{\otimes n}).
\]
Taking the $n$th root, taking the logarithm $\log_2$ and letting $n$ go to infinity, gives
\[
\subomega(\phi) \leq \subomega(A_\phi) = \log_2 \subrank(A_\phi) =  \log_2 \rank(A_\phi) = \omega(A_\phi) \leq \omega(\phi),
\]
by \cref{limprop}, finishing the proof.
\end{proof}

\begin{example}\label{graphcut}
A cut of a graph $G=(V,E)$ is a partition of $V$ into two disjoint nonempty sets. A max-cut is a cut with maximal number of edges crossing these two sets. A min-cut is a cut with minimal number of edges crossing these two sets. For any graph $G$, let $f(G)$ denote the size of a max-cut of $G$, and $g(G)$ the size of a min-cut. Flattening along the appropriate cuts yields
\[
\subrank(\tens(G)) \leq 2^{g(G)}\leq 2^{f(G)} \leq \rank(\tens(G))
\]
and thus
\begin{equation}\label{cutseq}
\subomega(\tens(G)) \leq g(G)\leq f(G) \leq \omega(\tens(G)).
\end{equation}
While for a 2-tensor $\phi$, we have $\subrank(\phi) = \rank(\phi)$ and $\subomega(\phi) = \omega(\phi)$, the above observation gives us many examples of higher-order tensors for which this is not the case. For example, let $\phi = \tens(C_k)$. Then $g(C_k) = 2$ and $f(C_k) = k-1$.
\end{example}

Let us discuss the exponent in more detail.

\begin{definition}%
Define the real number $\omega \coloneqq \omega(\tens(C_3))$. This number is called the \emph{exponent of matrix multiplication}.
\end{definition}

In algebraic complexity theory, the number $\omega$ is a measure for the asymptotic complexity of multiplying two $n\times n$ matrices, and has been receiving much attention since the discovery of Strassen's matrix multiplication algorithm \cite{strassen1969gaussian}. See also the standard reference for algebraic complexity theory \cite{burgisser1997algebraic}.
The following bounds on $\omega$ are the state of the art. The upper bound is by Le Gall \cite{le2014powers} and the lower bound is by a standard flattening argument (\cref{graphcut}).
\begin{theorem}\label{thmle2014powers}
$2 \leq \omega \leq 2.3728639$.
\end{theorem}

Computing $\omega(\tens(G))$ will be a hard task in general, since it includes computing the matrix multiplication exponent $\omega$. One may however try
 to prove bounds on $\omega(\tens(G))$ \emph{in terms of} the matrix multiplication exponent $\omega$ or the \emph{dual exponent} $\alpha$.  %
The number $\alpha$ is defined as ${\sup\{\gamma \in \RR \mid \omega(1,1,\gamma) = 2\}}$, where $\omega(1,1,\gamma)$ is defined as $\inf\{\beta \in \RR \mid \rank(\langle n, n, \floor{n^{\gamma}}) = \Oh(n^\beta)\}$ and for $n_1, n_2, n_3\in \NN$ the tensor $\langle n_1, n_2, n_3\rangle$ is the rectangular matrix multiplication tensor
\begin{multline*}
\langle n_1,n_2,n_3 \rangle \coloneqq \sum_{\mathclap{\substack{i \in [n_1]\times [n_2]\times [n_3]}}} (b_{i_1}\otimes b_{i_2})\otimes (b_{i_2}\otimes b_{i_3}) \otimes (b_{i_3} \otimes b_{i_1})\\[-0.5em]
\in (\CC^{n_1}\otimes \CC^{n_2}) \otimes (\CC^{n_2}\otimes \CC^{n_3}) \otimes (\CC^{n_3} \otimes \CC^{n_1}).
\end{multline*}
The  state of the art bounds on $\alpha$ are $0.3029805 < \alpha \leq 1$ \cite{le2012faster}.
In \cite{christandl2016tensor} we used a technique called tensor surgery to obtain such a result for cycle graphs $C_k$. Namely, 
let $k \in \NN_{\geq 1}$. Then
\begin{alignat}{2}
&\omega(\tens(C_k)) = k  \qquad&&\textnormal{when $k$ is even,}\\
k-1 \leq{}& \omega(\tens(C_k)) \leq \frac{k-1}{2} \omega  \qquad&&\textnormal{when $k$ is odd.}\label{eqoddcycle}
\end{alignat}
Also, in terms of the dual exponent $\alpha$, for odd $k$,
\[
\omega(T(C_k)) \leq k - \alpha\Bigl( 1 + \frac{1-\alpha}{k-1+\alpha} \Bigr) \leq k-\alpha.
\]
In particular, if $\omega = 2$ (equivalently $\alpha = 1$), then $\omega(\tens(C_k)) = k-1$ when $k$ is odd.
We will see a similar type of result for complete graphs in the next subsection, see~\eqref{trianglecover}. We note that tensor surgery seems to work well for sparse graphs like cycle graphs, and not so well for dense graphs like the complete graph. However, one could use the complete graph tensors as starting tensors in a tensor surgery procedure.

\begin{definition}
For any graph $G = (V,E)$, we define the \emph{exponent per edge}
\[
\tau(\tens(G)) \coloneqq \omega(\tens(G))/|E|.
\]
The letter $\tau$ is borrowed from the Schönhage asymptotic sum inequality, also known as the $\tau$-theorem.
\end{definition}

For complete graphs, the exponent per edge has the following monotonicity property.

\begin{proposition}\label{peredge}
Let $k\geq\ell\geq 2$ be integers. Then $\tau(\tens(K_{k})) \leq \tau(\tens(K_{\ell}))$.
\end{proposition}
\begin{proof}
Let $\tau_\ell = \tau(\tens(K_\ell))$. Label the vertices of $K_k$ by $1,\ldots, k$. Then, for any subgraph $G$ in $K_k$ isomorphic to~$K_\ell$ we have, with subscripts denoting tensor leg positions,
\[
\bigl((\tens(G))_{V(G)} \otimes (1\otimes \cdots \otimes 1)_{[k]\setminus V(G)}\bigr)^{\otimes N} \leq \tens^{\otimes \bigl(\binom{\ell}{2}\tau_\ell N + o(N)\bigr)}
\]
and so
\[
\bigotimes_{\substack{G\subseteq K_k:\\ G\cong K_\ell}} \bigl((\tens(G))_{V(G)} \otimes (1\otimes \cdots \otimes 1)_{[k]\setminus V(G)}\bigr)^{\otimes N} \leq \tens^{\otimes \bigl(\binom{\ell}{2}\tau_\ell N + o(N)\bigr)\binom{k}{\ell}},
\]
the tensor product taken over all subgraphs $G$ in $K_k$ isomorphic to $K_\ell$.
The left-hand side is isomorphic to $\tens(K_k)^{\otimes {N\binom{k-2}{\ell-2}}}$. Therefore, we have the upper bound $\omega(\tens(K_k)) \leq \binom{\ell}{2}\binom{k}{\ell}\binom{k-2}{\ell-2}^{-1} \tau_\ell = \binom{k}{2}\tau_\ell$, so $\tau(\tens(K_k)) \leq \tau_\ell$.
\end{proof}

Having discussed the exponent, we now wish to focus on the (monomial) subexponent, and primarily on an important result about the (monomial) subexponent of so-called tight 3-tensors.

\begin{definition}\label{supportdef} Let $\phi$ be an element of $V_1 \otimes \cdots \otimes V_k$.
Let $B_1, \ldots, B_k$ be bases for $V_1, \ldots, V_k$ respectively. Write $\phi$ in terms of these bases,
\[
\phi = \sum_{b_i\in B_i} \phi(b_1, \ldots, b_k)\, b_1 \otimes \cdots \otimes b_k\quad \textnormal{with}\quad \phi(b_1, \ldots, b_k) \in \CC.
\]
Then the \defin{support} of $\phi$ with respect to $B=(B_1, \ldots, B_k)$ is defined as the set of $k$-tuples
\[
\supp_B \phi \coloneqq \{(b_1, \ldots, b_k)\in B_1\times \cdots \times B_k \mid  \phi(b_{1}, \ldots, b_{k})\neq 0\}.
\]
When the basis is clear from the context we will simply write $\supp \phi$. It is often convenient to identify each $B_i = \{b_1, b_2, \ldots, b_{\abs[0]{B_i}}\}$ with its index set $\{1,2,\ldots, \abs[0]{B_i}\}$ by $b_j \mapsto j$ so that $\supp_B \phi$ becomes a set of tuples of natural numbers. (Sometimes the index set starts at 0 instead of 1.)
\end{definition}

\begin{definition}\label{deftight}
Let $B_1,\ldots, B_k$ be bases for $V_1, \ldots, V_k$ respectively and let $\alpha_1: B_1 \to \ZZ,\, \ldots,\, \alpha_k: B_k \to \ZZ$ be injective maps such that 
\[
\alpha_1(b_{1}) + \cdots + \alpha_k(b_{k}) = 0\quad \textnormal{for each tuple}\quad (b_1,\ldots, b_k) \in \supp \phi.
\]
Then we say $\phi$ is \defin{tight with respect to $\alpha = \alpha_1\times \cdots \times \alpha_k$}. We say $\phi$ is \defin{tight} if it is tight with respect to $\alpha = \alpha_1\times \cdots \times \alpha_k$ for some $\alpha_i : B_i \to \ZZ$.
\end{definition}

The following lemma is easy to prove.

\begin{lemma}\label{prodtight}
Tensor products of tight tensors are tight in the tensor product basis. Therefore, any graph tensor $\tens(G)$ is tight in the standard basis.
\end{lemma}

\begin{example}\label{tightexample}
$\tens(C_k) = \sum_{i\in [2]^k} (b_{i_1}\otimes b_{i_2}) \otimes (b_{i_2}\otimes b_{i_3}) \otimes \cdots \otimes (b_{i_k}\otimes b_{i_1})$ is tight for any $k$ by \cref{prodtight}.
Explicitly, let 
\begin{alignat*}{3}
\alpha_1&:{}& &b_{i_1} \otimes b_{i_2} &&\mapsto i_1 - 2(i_2-1),\\
\alpha_2&:{}& &b_{i_2} \otimes b_{i_3} &&\mapsto 2(i_2-1) - 4(i_3-1),\\
\vdots\\
\alpha_k&:{}& &b_{i_k} \otimes b_{i_1} &&\mapsto 2^{k-1}(i_1-1) - i_1.
\end{alignat*}
These maps are injective and $\alpha_1(b_{i_1} \otimes b_{i_2}) + \cdots + \alpha_k(b_{i_k} \otimes b_{i_1}) = 0$ for any~$i \in [2]^k$.
\end{example}

\begin{example}\label{Dtight} For any partition $\lambda\vdash_d k$, the Dicke tensor $D_{\lambda}$ is a tight tensor. Namely, let $s = \sum_{j=1}^d \lambda_j \cdot j$, and let
\begin{alignat*}{3}
\alpha_1&:{}& &b_{i} &&\mapsto i,\\
\alpha_2&:{}& &b_{i} &&\mapsto i,\\
\vdots\\
\alpha_k&:{}& &b_{i} &&\mapsto i - s.
\end{alignat*}
Then $\alpha_1(b_1) + \cdots + \alpha_k(b_k) = 0$ for any $(b_1, \ldots, b_k) \in \supp D_\lambda$.
\end{example}

For tight 3-tensors, Strassen proved the following theorem. For a discrete probability distribution $P = (p_1, \ldots, p_n)$, let $H(P)$ be the \defin{Shannon entropy} of $P$, which is defined by $H(P) \coloneqq -\sum_{i=1}^n p_i \log_2 p_i$. Recall that if~$P$ is a probability distribution on a set $X$ consisting of $k$-tuples $(x_1, \ldots, x_k)$, then the marginal distribution $P_i$ on the $i$th component $X_i = \{x_i \mid x\in X\}$ is defined by $P_i(y) \coloneqq \sum_{x\in X:\, x_i = y} P(x)$.

\begin{theorem}[\cite{strassen1991degeneration}]\label{strassentripartite} Let $\phi$ be a 3-tensor which is tight in some basis~$B$. Then
\begin{equation}\label{strassen}
\combsubomega(\phi) = \subomega(\phi) = \max_{P\in\optset{P}_\Phi} \min\{H(P_1), H(P_2), H(P_3)\},
\end{equation}
where $\optset{P}_\Phi$ consists of the probability distributions $P$ on $\supp_B \phi$, and $P_i$ is the marginal distribution of $P$ on the $i$th component of $\supp_B \phi$.
\end{theorem}

Interestingly, \cref{strassentripartite} says that if $\phi$ is a tight tensor in basis~$B$, then the subexponent $\subomega(\phi)$ is a function of $\supp_B \phi$, that is, $\subomega(\phi)$ is independent of the coefficients of $\phi$ in basis $B$.
 This is not in general the case for higher-order tight tensors as we will see in \cref{cex}.

\begin{example}
From \cref{Dtight} we know that
$D_{(1,1,1)}$ is tight. Let~$P$ be the uniform distribution on the support of $D_{(1,1,1)}$. Then each marginal~$P_i$ is uniform on $\{b_1, b_2, b_3\}$. Therefore by \cref{strassentripartite}, we have $\subomega( D_{(1,1,1)}) \geq \combsubomega( D_{(1,1,1)}) \geq \log_2 3$. On the other hand,  $\subomega(D_{(1,1,1)}) \leq \log_2 3$, because the tensor $D_{(1,1,1)}$ has a flattening of rank 3.
\end{example}

\begin{example} From \cref{prodtight} and \cref{tightexample} we know that $\phi = \tens(C_3)$ is tight.
Let~$P$ be the uniform distribution on 
\[
\supp \phi = \{((b_i\otimes b_j), (b_j \otimes b_k), (b_k \otimes b_i))\mid i,j,k\in [2]\}.
\]
Then the marginals $P_i$ are uniform on $\{b_i \otimes b_j \mid i,j \in [2]\}$, and hence the Shannon entropy of each marginal is $H(P_i) = \log_2 4 = 2$. Therefore, by \cref{strassentripartite}, $\subomega( \tens(C_3) ) \geq 2$. On the other hand, $\subomega(\tens(C_3)) \leq 2$, because the tensor $\tens(C_3)$ has a flattening of rank 4.
\end{example}

There are two extensions of \cref{strassentripartite} to certain higher-order tensors, namely to W-state tensors (\cref{wstate}) and to (hyper)graphs with a certain connectedness property (\cref{hyper}), which includes cycle tensors. We begin with the extension to W-state tensors, obtained by generalising a construction of Coppersmith and Winograd and generalising Strassen's support functionals. Define the \defin{binary entropy function} $h(p)$ as $-p\log_2(p) - {(1-p)}\log_2(1-p)$ with $h(0) = h(1) = 0$.

\begin{theorem}[Vrana--Christandl~\cite{vrana2015asymptotic}]\label{wstate} Let $k\geq 2$. Then the subexponent of $D_{(1,k-1)}$ satisfies
$\subomega(D_{(1,k-1)}) = \combsubomega(D_{(1,k-1)})= h(1/k)$.
\end{theorem}

(The original proof contained a small mistake for which we provide a fix in \cref{appl}.)
It is conjectured in \cite{vrana2015asymptotic} that in general $\subomega(D_{\lambda}) = \combsubomega(D_{\lambda})= H(\lambda/k)$, see also \cref{q3}.

\begin{remark}
In \cite{fu} the Coppersmith--Winograd construction for $k=3$ was used to prove lower bounds on the query complexity of testing triangle-freeness of Boolean functions. This construction was later extended to cover odd-cycle-freeness \cite{haviv}, thus independently proving part of the statement of \cref{wstate}.
\end{remark}

\begin{example}
Let $\phi \in (\CC^2)^{\otimes 3}$ be nonzero. Then one of the following statements holds.
\def\grsizex{1.2em}
\begin{enumerate}
\item $\phi \cong \tens$ and $\subomega(\phi) = 1$.
\item $\phi \cong D_{1,2}$ and $\subomega(\phi) = h(1/3)$.
\item $\phi \cong \tens\bigl(\pbox{4em}{
\begin{tikzpicture}[line width=0.2mm, vertex/.style={
    circle,
    fill=white,
    draw,
    outer sep=0pt,
    inner sep=0.1em}, x=\grsizex, y=\grsizex]
    \path[coordinate] (0,0)  coordinate(A) (1,0) coordinate(B) (2,0) coordinate(C);
	\draw [line width=0.2mm]  (A) node [vertex] {} -- (B) node [vertex] {} (C) node [vertex] {} ;
\end{tikzpicture}}\bigr)$, $\tens\bigl(\pbox{4em}{
\begin{tikzpicture}[line width=0.2mm, vertex/.style={
    circle,
    fill=white,
    draw,
    outer sep=0pt,
    inner sep=0.1em}, x=\grsizex, y=\grsizex]
    \path[coordinate] (0,0)  coordinate(A) (1,0) coordinate(B) (2,0) coordinate(C);
	\draw [line width=0.2mm]  (A) node [vertex] {} edge[bend left] (C) node [vertex] {} (B) node [vertex] {} ;
	\draw (C) node [vertex] {};
\end{tikzpicture}}\bigr)$, $\tens\bigl(\pbox{4em}{
\begin{tikzpicture}[line width=0.2mm, vertex/.style={
    circle,
    fill=white,
    draw,
    outer sep=0pt,
    inner sep=0.1em}, x=\grsizex, y=\grsizex]
    \path[coordinate] (0,0)  coordinate(A) (1,0) coordinate(B) (2,0) coordinate(C);
	\draw [line width=0.2mm]  (A) node [vertex] {}  (B) node [vertex] {} -- (C) node [vertex] {} ;
\end{tikzpicture}}\bigr)$ or $\tens\bigl(\pbox{4em}{
\begin{tikzpicture}[line width=0.2mm, vertex/.style={
    circle,
    fill=white,
    draw,
    outer sep=0pt,
    inner sep=0.1em}, x=\grsizex, y=\grsizex]
    \path[coordinate] (0,0)  coordinate(A) (1,0) coordinate(B) (2,0) coordinate(C);
	\draw [line width=0.2mm]  (A) node [vertex] {}  (B) node [vertex] {} (C) node [vertex] {} ;
\end{tikzpicture}}\bigr)$ and $\subomega(\phi) = 0$.
\end{enumerate}
Indeed, it is well-known that any element $\phi \in (\CC^2)^{\otimes 3}$ is equivalent to precisely one of the six tensors listed above, see  \cite{dur2000three} where $\cong$ is called ``equivalence under SLOCC''. In the terminology of the reference, the first tensor is in the GHZ-class, the second tensor is in the W-class and the remaining classes are called A-BC, AB-C, AC-B and A-B-C. In the first case, by definition $\subomega(\phi) = \omega(\tens, \tens)^{-1} = 1$. In the second case, $\subomega(\phi) = \subomega(D_{1,2}) = h(1/3)$ by \cref{wstate}. In the third case, every representative corresponds to a graph with min-cut of size 0, so $\subomega(\phi) \leq g(\phi) = 0$ (see \cref{graphcut}).
\end{example}

Recall that we defined $g(G)$ to be the size of a minimum cut of $G$ in \cref{graphcut}.
The following theorem is a special case of the result proved in \cite{vrana}.

\begin{theorem}[Vrana--Christandl~\cite{vrana}]\label{hyper}
Let $G$ be a graph. Then
\[
\combsubomega(\tens(G)) = \subomega(\tens(G)) = g(G).
\]
In particular, if $k\geq 3$, then the subexponent of $\tens(C_k)$ satisfies $\subomega(\tens(C_k)) = 2$.
\end{theorem}

\begin{example}\label{cex} Strassen's \cref{strassentripartite} does not generalize to $k\geq4$ by simply replacing the right-hand side of \eqref{strassen} by $\max_P \min\{H(P_1), \ldots, H(P_k)\}$ as the following example shows. Let $\phi$ be the 4-tensor
\def\grsizex{1.2em}
\begin{align*}
\phi &= \tens\bigl(\pbox{4em}{
\begin{tikzpicture}[line width=0.2mm, vertex/.style={
    circle,
    fill=white,
    draw,
    outer sep=0pt,
    inner sep=0.1em}, x=\grsizex, y=\grsizex]
    \path[coordinate] (0,0)  coordinate(A) (1,0) coordinate(B) (2,0) coordinate(C) (3,0) coordinate(D);
	\draw [line width=0.2mm]  (A) node [vertex] {} -- (B) node [vertex] {} (C) node [vertex] {} -- (D) node [vertex] {} ;
\end{tikzpicture}}\bigr) \\
&=(b_0 \!\otimes\! b_0 + b_1\!\otimes\! b_1) \otimes (b_0\!\otimes\!b_0 + b_1\!\otimes\!b_1)\\
 &= b_{0}\!\otimes\! b_{0}\!\otimes\! b_{0}\!\otimes\! b_{0} + b_{0}\!\otimes\! b_{0}\!\otimes\! b_{1}\!\otimes\! b_{1}\\ &\quad+ b_{1}\!\otimes\! b_{1}\!\otimes\! b_{0}\!\otimes\! b_{0} + b_{1}\!\otimes\! b_{1}\!\otimes\! b_{1}\!\otimes\! b_{1}.
\end{align*}
This tensor is tight. Take for example
\begin{alignat*}{2}
\alpha_1 : b_{i} \mapsto \phantom{-}i, &\qquad&\alpha_3 : b_{i} \mapsto \phantom{-}i,\\
\alpha_2 : b_{i} \mapsto -i, &\qquad&\alpha_4 : b_{i} \mapsto -i.
\end{alignat*}
Let $P$ be the uniform distribution on $\supp \phi$. Its marginals $P_i$ are uniform on $\{b_{0},b_{1}\}$ so $H(P_i) = 1$. However, $\subomega(\phi) = 0$, since $\phi$ by construction has a flattening of rank 1.

This example also shows that $\subomega(\phi)$ cannot simply be a function of the support of $\phi$, since, for $0< p < 1$, the tensor
\[
p(b_{0}\!\otimes\! b_{0}\!\otimes\! b_{0}\!\otimes\! b_{0} + b_{0}\!\otimes\! b_{0}\!\otimes\! b_{1}\!\otimes\! b_{1}) + b_{1}\!\otimes\! b_{1}\!\otimes\! b_{0}\!\otimes\! b_{0} + b_{1}\!\otimes\! b_{1}\!\otimes\! b_{1}\!\otimes\! b_{1}
\]
has no flattenings of rank at most $1$, and hence $\subomega(\phi)$ is strictly positive (see Lemma~4 in \cite{vrana2015asymptotic}).
\end{example}

\begin{remark}
\cref{cex} suggests that to generalize \cref{strassentripartite} to $k$-tensors $\phi$ one either has to find a stronger condition when $k\geq4$ to guarantee that $\subomega(\phi)=\max_P \min\{H(P_1), \ldots, H(P_k)\}$, or find a different lower (and possibly upper) bound on the subexponent~$\subomega(\phi)$. Our \cref{main} is a result of the second type.
\end{remark}

\subsection{Our results}

This paper is motivated by the following problem on complete graph tensors. Recall that we defined the \emph{exponent per edge} for the complete graph tensor as $\tau(\tens(K_k)) = \omega(\tens(K_k))/\binom{k}{2}$

\begin{problem}
For $k\geq 4$, what is the value of $\tau(\tens(K_k))$?
\end{problem}

First of all, it is not hard to prove the following bounds on $\tau(\tens(K_k))$. For the complete graph $K_k$, the maximum cut size $f(K_k)$ is $k^2/4$ for even~$k$ and $(k-1)(k+1)/4$ for odd~$k$ (maximum cut is defined in  \cref{graphcut}). Then, flattening $\tens(K_k)$ along a max-cut yields a matrix of rank $2^{f(K_k)}$ and therefore (\cref{flat})
\begin{equation}\label{flateq}
\tau(\tens(K_k)) \geq \frac{f(K_k)}{\binom{k}{2}} = \begin{cases}
1/2+1/(2k) & \textnormal{for odd $k$}\\
1/2+1/(2(k-1)) & \textnormal{for even $k$}.
\end{cases}
\end{equation}
On the other hand, $\tau(\tens(K_3)) = \omega/3$ and thus by \cref{peredge} we have $\tau(\tens(K_k)) \leq \omega/3$ for all $k \geq 3$.
Plugging in the Le Gall upper bound $\omega \leq 2.3728639$ (\cref{thmle2014powers}), yields the ``triangle covering'' upper bound 
\begin{equation}\label{trianglecover}
\tau(\tens(K_k)) \leq 0.790955.
\end{equation}
Our aim is to improve on this upper bound.

Our main result is an upper bound on the exponent per edge of $\tens(K_k)$ that is independent of $\omega$.

\begin{theorem}\label{cwcomplete} Let $k\geq 2$.  For any $q\geq 1$, 
$\tau(\tens(K_k)) \leq \log_q \Bigl(\displaystyle \frac{q + 2}{2^{\combsubomega(D_{(2,k-2)})}} \Bigr)$.
\end{theorem}

Our second result is a lower bound on the monomial subexponent~$\combsubomega(\phi)$ of any tight tensor $\phi$. %

\begin{definition}
Let $\phi$ be a $k$-tensor, $B$ a basis, $\Phi$ the corresponding support as in \cref{supportdef}. Suppose $\phi$ is tight with respect to $\alpha$. For $R\subseteq \Phi\times \Phi$, define $r_\alpha(R)$ to be the rank of the matrix with rows
\[
\{\alpha(x) - \alpha(y) \mid (x,y)\in R\}.
\]
We will denote $r_\alpha(R)$ by $r(R)$ when the actual~$\alpha$ is clear or not important. For any $i\in \{1,\ldots, k\}$, define
\[
\mathcal{R}_i \coloneqq \{(x,y) \in \Phi\times \Phi \mid x_i = y_i\}.
\]
\end{definition}

\begin{theorem}\label{main}
Let $\phi$ be a $k$-tensor, $B$ a basis, $\Phi$ the corresponding support. Suppose $\phi$ is tight with respect to $\alpha$.
Then,
\begin{equation}\label{mainineq}
\combsubomega(\phi) \geq \max_{P\in\optset{P}_{\Phi}}\! \biggl(\!H(P) - (k-2) \max_{\mathclap{R\in \optset{R}_{\Phi}}} \frac{\max_{Q\in \optset{Q}_{\Phi,(P_1, \ldots, P_k),R}} H(Q) - H(P)}{r_\alpha(R)} \biggr),
\end{equation}
where
\begin{itemize}
\item $\optset{P}_{\Phi}$ consists of all probability distributions $P$ on $\Phi$;  and $P_1, \ldots, P_k$ are the marginal distributions of $P$ on the $k$ components respectively;
\item $\optset{R}_{\Phi}$ consists of all subsets $R \subseteq \Phi \times \Phi$ that are not contained in the diagonal set $\Delta_{\Phi}\coloneqq \{(x,x) \mid x\in \Phi\}$, and such that $R\subseteq \mathcal{R}_i$ for some $i\in \{1,\ldots,k\}$.
\item $\optset{Q}_{\Phi, (P_1, \ldots, P_k),R}$ consists of all probability distributions $Q$ on $R$ whose marginals on the $2k$ components of~$R$ satisfy
$Q_i = Q_{k+i} = P_i$ for $i \in \{1,\ldots, k\}$. %
\end{itemize}
\end{theorem}

(The symbols $\optset{P}$, $\optset{R}$, $\optset{Q}$ are script versions of the letters $P$, $R$ and $Q$.)

The lower bound in \cref{main} extends the lower bound in \cref{strassentripartite}. We will prove this in \cref{appl}.

We will use \cref{main} to compute the (monomial) subexponent of the weight-$(2,2)$ Dicke tensor $D_{(2,2)}$. We give the proof in \cref{appl}.

\begin{corollary}\label{dickecor}
$\combsubomega(D_{(2,2)}) = \subomega(D_{(2,2)}) = 1$.
\end{corollary}

\cref{peredge}, \cref{cwcomplete} with $k=4$ and \cref{dickecor} together directly imply the following upper bound on the exponent per edge of the complete graph tensor $\tens(K_k)$ for any $k\geq4$.

\begin{corollary}\label{cor}
For $k\geq 4$,
$\tau(\tens(K_k)) \leq  \min_{q\geq 2} \log_q \Bigl(\displaystyle \frac{q + 2}{2} \Bigr)=\log_7(9/2)$ which is approximately $0.772943$. %
\end{corollary}

Note that $0.772943$ is strictly smaller than the triangle cover upper bound $0.790955$ from \eqref{trianglecover}. 

In the following table we summarize, for small $k$: the flattening lower bound on $\omega(\tens(K_k))$; the upper bound from \cref{cor} on $\omega(\tens(K_k))$ (and Le Gall's upper bound for $k=3$); the trivial upper bound on $\omega(\tens(K_k))$ given by the number of edges; and the resulting bounds on~$\tau(\tens(K_k))$.

\begin{center}
\begin{tabular}{llllllll}
\toprule
$k$ &  \multicolumn{2}{c}{$\omega(\tens(K_k))$} && $\binom{k}{2}$ && \multicolumn{2}{c}{$\tau(\tens(K_k))$}  \\\cmidrule{2-3} \cmidrule{7-8}
    & lower & upper && && lower & upper     \\
\midrule
3 & 2  & 2.37287 && 3  && 0.666666 & 0.790955\\
4 & 4  & 4.63766 && 6  && 0.666666 & 0.772943\\
5 & 6  & 7.72943 && 10 && 0.6      & 0.772943\\
6 & 9  & 11.5942 && 15 && 0.6      & 0.772943\\
7 & 12 & 16.2319 && 21 && 0.571428 & 0.772943\\
8 & 16 & 21.6425 && 28 && 0.571428 & 0.772943\\
9 & 20 & 27.8260 && 36 && 0.555555 & 0.772943\\
10& 25 & 34.7825 && 45 && 0.555555 & 0.772943\\
\bottomrule
\end{tabular}
\end{center}

\subsection{Questions}\label{subsecquestions}

We discuss three open questions related to our results. Our first question is about complete graph tensors. From \eqref{flateq} we know that
\[
\tau(\tens(K_{k+1})) \geq \tau(\tens(K_k)) \geq \frac{(k-1)(k+1)}{4\binom{k}{2}}
\]
holds for all odd $k$. The lower bound goes to $\tfrac12$ when $k$ goes to infinity. On the other hand, if $\omega=2$, then $\tau(\tens(K_k)) \leq \tfrac23$ holds for all~$k\geq 3$ (\cref{peredge}), and in particular $\tau(\tens(K_3)) = \tau(\tens(K_4)) = \tfrac23$. We thus ask the following question.

\begin{question}\label{q1}
Is there a $k\geq 5$ such that $\tau(\tens(K_k)) < \tfrac23$?
\end{question}

Our second question is a more precise version of the above question for general graph tensors.

\begin{question}\label{a2}
Is it true that for every graph $G$, we have $\omega(\tens(G)) = f(G)$?
\end{question}

The answer is known to be yes only for bipartite graphs. (This follows directly from the lower bound in \eqref{cutseq}.) If the answer to \cref{a2} is yes, then the answer to \cref{q1} is yes. For $G = C_3$, this question specializes to the long-standing open question whether the exponent of matrix multiplication $\omega$ equals 2. It might therefore be interesting to ask \cref{a2} with the additional assumption that $\omega=2$. The answer is then known to be yes for all odd cycles (see \eqref{eqoddcycle}), the bipartite graphs, the complete graph~$K_4$ and graphs that are composed of these in a certain (natural) way; namely by taking the disjoint union and then identifying pairs of nonadjacent vertices of choice.

Our third question is about the (monomial) subexponent of Dicke tensors. In \cite{vrana} the upper bound $\subomega(D_{\lambda}) \leq H(\lambda/k)$ was proven for any $\lambda\vdash k$, and the following problem was posed.

\begin{question}\label{q3}
Is it true that for every partition $\lambda\vdash k$, 
\begin{equation}\label{dickeconj}
\combsubomega(D_{\lambda}) = \subomega(D_\lambda) = H(\lambda/k)?
\end{equation}
\end{question}

For $k=3$ this question was affirmatively answered in \cite{strassen1991degeneration} (see \cref{strassentripartite} above), and 
for $\lambda = (k-1,1)$ and any $k\geq 3$ this question was affirmatively answered in \cite{vrana} (see \cref{wstate} above). \cref{main} extends both results and covers more cases, including $\lambda=(2,2)$, see \cref{dickecor}. %
We also numerically confirmed \eqref{dickeconj} for $\lambda = (\ell, \ell)$ and $2\leq \ell \leq 1000$ using \cref{main}. We conjecture that \cref{main} is strong enough to prove the lower bound $\combsubomega(D_\lambda) \geq H(\lambda/k)$ for any~$\lambda \vdash k$.

\subsection{Outline} The rest of this paper is organized as follows. In \cref{seccwcomplete} we will prove \cref{cwcomplete} on the exponent of $\tens(K_k)$. In \cref{secmain} we will prove \cref{main} on the monomial subexponent of tight tensors and we will compute the monomial subexponent of the weight-$(2,2)$ Dicke tensor.

\section{Upper bound on the exponent of the complete graph tensor}\label{seccwcomplete}

The main structure of the proof of \cref{cwcomplete} is a generalization of a construction of Strassen \cite{Strassen:1986:AST} which was improved by Coppersmith and Winograd \cite{coppersmith1987matrix}, which involves finding a ``cheap'' starting tensor, generalizing Schönhage's asymptotic sum inequality \cite{schonhage1981partial} and choosing a good block decomposition.

\subsection{Preliminaries}

We start by discussing border rank in more detail, then we introduce the generalized asymptotic sum inequality and finally we review block decompositions of tensors.

Let $\phi \in V_1\otimes \cdots \otimes V_k$ be a tensor. %

\begin{definition}
Let $\rank_h(\phi)$ be the smallest number $r$ such that there are matrices $A_1(\varepsilon) : \CC^r \to V_1$, \ldots, $A_k(\varepsilon) : \CC^r \to V_k$ with entries in $\CC[\varepsilon]$, such that $(A_1(\varepsilon) \otimes \cdots \otimes A_k(\varepsilon)) \tens_r = \varepsilon^h \phi + \Oh(\varepsilon^{h+1})$. Here, $\Oh(\varepsilon^{h+1})$ denotes any expression of the form $\sum_{i=1}^\ell \varepsilon^{h+i} \phi_i$ with $\phi_i \in V_1\otimes \cdots \otimes V_k$ for some $\ell$.
\end{definition}

\begin{theorem}[{\cite[Theorem~20.24]{burgisser1997algebraic}}]\label{minh}
$\borderrank(\phi) = \min_h \rank_h(\phi)$.
\end{theorem}

It is well-known that rank and border rank are related as follows.

\begin{proposition}\label{borderranktorank}
Let $k,h\in \NN$. Let $c_h = \binom{h+k-1}{k-1}$. 
For all $m_1, \ldots, m_k\in \NN$ and all tensors $\phi\in \CC^{m_1} \otimes \cdots \otimes \CC^{m_k}$, we have $\rank(\phi) \leq c_h \rank_h(\phi)$. Note that for fixed $k$, the number $c_h$ is upper bounded by a polynomial in $h$; $c_h \leq (h+1)^{k-1}$.
\end{proposition}
\begin{proof}
Let $\phi$ be a tensor in $\CC^{m_1}\otimes \cdots \otimes \CC^{m_k}$ with $\rank_h(\phi) = r$. Then there are vectors $v_i^j \in (\CC[\varepsilon])^{m_j}$ such that
\[
\sum_{i=1}^r v_i^1 \otimes \cdots \otimes v_i^k = \varepsilon^h \phi + \Oh(\varepsilon^{h+1}).
\]
Without loss of generality the highest power of $\varepsilon$ in each $v_i^j$ is $h$.
Decomposing every $v_i^j$ into $\varepsilon$-homogeneous components $v_i^j = \sum_{a_j=0}^h \varepsilon^{a_j} v_i^j(a_j)$, and collecting powers of $\varepsilon$ gives
\[
\sum_{i=1}^r \sum_{a \in [h]^k} \!\varepsilon^{a_1 + \cdots + a_k}\hspace{0.5em} v_i^1(a_1) \otimes \cdots \otimes v_i^k(a_k) = \varepsilon^h \phi + \Oh(\varepsilon^{h+1}).
\]
Taking only the summands such that $a_1 + \cdots + a_k = h$ gives a rank decomposition of $\phi$. There are $\binom{h+k-1}{k-1} r$ such summands. Therefore, the statement of the proposition holds for $c_h = \binom{h+k-1}{k-1}$ which is at most $(h+1)^{k-1}$.
\end{proof}

As a consequence of \cref{borderranktorank} we can upper bound the exponent of a tensor by the border rank of that tensor. (The following proposition also follows from the more general \cref{restrdegen}.)
\begin{proposition}\label{exponentborder}
Let $\psi$ be a tensor. Then
\[
\omega(\psi) \leq \log_2 \borderrank(\psi).
\]
\end{proposition}
\begin{proof}
Suppose $\borderrank(\psi) = r$. Let $h\in \NN$ such that $\rank_h(\psi) = r$ (\cref{minh}). Then for all~$s\in \NN$, $\rank_{hs}(\psi^{\otimes s}) \leq r^s$. Therefore, $\rank(\psi^{\otimes s}) \leq c_{hs}r^s$ (\cref{borderranktorank}). Thus the exponent $\omega(\psi)$ is at most $\tfrac{1}{s} \log_2 r^s + \tfrac{1}{s} \log_2 c_{hs}$, which converges to $\log_2 r$ when $s$ goes to infinity.
\end{proof}

Let $G = (V,E)$ be a graph and let $f: E\to \NN$ be a function that assigns to every edge a natural number. We define a ``nonuniform'' version of $\tens(G)$ as follows. Define $\tens_f(G)$ by
\[
\tens_f(G) \coloneqq \bigotimes_{e \in E} \sum_{i\in [f(e)]}  (b_i \otimes \cdots \otimes b_i)_e \otimes (1\otimes \cdots \otimes 1)_{V\setminus e}.
\]
Here, the large tensor product inputs and outputs $|V|$-tensors. In algebraic complexity theory, the tensor $\tens_f(C_3)$ with $f(1) = n_1, f(2) = n_2, f(3) = n_3$ is denoted by $\langle n_1, n_2, n_3\rangle$. This tensor corresponds to the bilinear map that multiplies an $n_1\times n_2$ matrix with an $n_2\times n_3$ matrix. We view the set $\{f:E\to \NN\}$ as a group under pointwise multiplication, so that we can write
\begin{equation}\label{kronprod}
\tens_f(G)\otimes \tens_g(G) \cong \tens_{fg}(G).
\end{equation}
Equation \eqref{kronprod} generalises the self-reducibility property of matrix multiplication tensors: $\langle m_1, m_2, m_3 \rangle \otimes \langle m_1, m_2, m_3\rangle \cong \langle m_1 n_1, m_2 n_2, m_3 n_3 \rangle$.

Recall that an \defin{automorphism} of a graph $G = (V,E)$ is a permutation~$\sigma$ of~$V$ such that for all $u,v\in V$ the pair $(u,v)$ is in $E$ if and only if $\sigma \cdot(u,v) \coloneqq (\sigma(u), \sigma(v))$ is in $E$. The automorphisms form a group $\Gamma$ under composition. The group~$\Gamma$ thus acts on $V$ and on $E$. We say $G$ is \defin{edge-transitive} if the action of $\Gamma$ on $E$ is transitive, meaning that for any two edges $e_1, e_2\in E$ there is a permutation $\sigma\in \Gamma$ such that $\sigma\cdot e_1 = e_2$. 

\begin{example}
The automorphism group of the cycle graph $C_k$ is the dihedral group with $2n$ elements. The automorphism group of the complete graph $K_k$ is the symmetric group $S_k$. Both graphs are edge-transitive.
\end{example}

\begin{theorem}[Generalized asymptotic sum inequality]\label{schonhage} Let $G = (V,E)$ be an edge-transitive graph. Suppose $r > p$. Suppose $\phi_1,\ldots, \phi_p$ are tensors from $\{\tens_f(G) \mid f:E\to\NN;\,  \prod_{e\in E} f(e) \geq 2\}$ such that 
\[
\borderrank(\phi_1\oplus \cdots \oplus \phi_p) \leq r.
\]
Define~$\tau$ by $\sum_{i=1}^p (\prod_{e\in E} f_i(e))^\tau = r$.
Then $\tau(\tens(G)) \leq \tau$.

In particular, if $\phi_1, \ldots, \phi_p \in \{\tens_f(G) \mid  f:E\to\NN;\, \prod_{e\in E} f(e) = q\}$ for some integer $q\geq 2$, then we have $\tau(\tens(G)) \leq \log_q (\borderrank(\phi_1 \oplus \cdots \oplus \phi_p)/p)$.
\end{theorem}
Our proof of \cref{schonhage} follows the structure of the proof for the cycle graph case in \cite{buhrman2016nondeterministic} which builds upon the exposition in \cite{blaser2013fast}.

\begin{proposition}\label{symmetrize}
Let $G=(V,E)$ be an edge-transitive graph. Let $f: E\to \NN$ be a function and $N=\prod_{e\in E} f(e)$. Assume that $N\geq 2$. Then,
\[
\tau(\tens(G) ) \leq \log_{N} \borderrank( \tens_f(G) ).
\]
\end{proposition}
\begin{proof}
Let $\Gamma$ be the automorphism group of $G$.
For any edge $e\in E$ define the stabilizer subgroup $\Gamma_e \coloneqq \{\sigma \in \Gamma \mid \sigma \cdot e = e\}$. Since $G$ is edge-transitive all subgroups $\Gamma_e$ are conjugates and thus have the same cardinality, say~$C$. This implies that, for any $e_1, e_2\in E$, the set $\{\sigma \in \Gamma \mid \sigma \cdot e_1 = e_2\}$ has cardinality~$C$. By the orbit-stabilizer theorem the number $C$ satisfies $|\Gamma| = C|E|$.
The group~$\Gamma$ acts naturally on $\tens_f(G)$ by permuting tensor legs. %
The tensor product of all elements $\sigma \cdot \tens_f(G)$ for $\sigma \in \Gamma$ equals $\tens_{N^C}(G)$. We thus have, using \cref{exponentborder},
\begin{align*}
\omega(\tens(G)) &= \frac{\omega(\tens_{N^C}(G))}{C \log_2 N} \\
&\leq \frac{\abs[0]{\Gamma} \log_2 \borderrank(T_f(G))}{C \log_2 N}\\
&= \abs[0]{E} \frac{\log_2 \borderrank(T_f(G))}{\log_2 N}
\end{align*}
finishing the proof. In the above we used $\omega(\tens_r(G)) = \omega(\tens(G)) \log_2 r$.
\end{proof}

We need the following upper bound on the rank of the $a$-fold direct sum of the $s$th tensor power of a tensor. 

\begin{lemma}\label{powertrick} Let $a,b\in \NN_{\geq 1}$.
Let $\phi$ be a tensor such that $\rank(\phi^{\oplus a}) \leq b$. Then for all $s\in \NN$, $\rank((\phi^{\otimes s})^{\oplus a}) \leq \lceil b/a \rceil^s a$.
\end{lemma}
\begin{proof}
We prove the lemma by induction over $s$. The base case $s=1$ follows from the assumption. For the induction step, we have
\[
(\phi^{\otimes (s+1)})^{\oplus a} = \phi^{\otimes s}\otimes \phi^{\oplus a} \leq  \phi^{\otimes s} \otimes \tens_b = (\phi^{\otimes s})^{\oplus b},
\]
and thus, by the induction hypothesis,
\begin{align*}
\rank((\phi^{\otimes (s + 1)})^{\oplus a}) &\leq \rank((\phi^{\otimes s})^{\oplus b})\\ 
&\leq \rank((\phi^{\otimes s})^{\oplus(a\lceil b/a\rceil)}) = \rank(((\phi^{\otimes s})^{\oplus a})^{\oplus \lceil b/a\rceil})\\ &\leq  \lceil\tfrac{b}{a}\rceil^s a \lceil \tfrac{b}{a} \rceil = \lceil\tfrac{b}{a}\rceil^{s+1} a,
\end{align*}
proving the lemma.
\end{proof}

\cref{powertrick} can equivalently be phrased as follows. Let $\phi$ be a tensor such that $\phi\otimes \tens_a \leq \tens_b$. Then for all $s\in \NN$ we have $\phi^{\otimes s} \otimes \tens_a \leq (\tens_{\ceil{b/a}})^{\otimes s} \otimes \tens_a$.

Formulated in the language of Strassen's semiring of $k$-tensors, \cref{powertrick} says: $a \cdot \phi \leq b \Rightarrow \forall s\, a\cdot \phi^s \leq \ceil{b/a}^s \cdot a$.

\cref{symmetrize} generalizes to the following inequality relating upper bounds on the exponent of a direct power $\tens_f(G)^{\oplus a}$ to upper bounds on the exponent of~$\tens(G)$.

\begin{proposition}\label{nonuniformsymm}
Let $G=(V,E)$ be an edge-transitive graph.
Let $f: E\to \NN$ be a function and $N=\prod_{e\in E} f(e)$. Assume that $N\geq 2$.
Then, $\tau(\tens(G)) \leq \log_{N} \lceil \rank(\tens_f(G)^{\oplus a})/a \rceil$, for any $a\geq 1$.
\end{proposition}
\begin{proof}
Let $b=\rank((\tens_f(G))^{\oplus a})$. Then by \cref{powertrick} we have the inequality $\rank((\tens_{f^s}(G))^{\oplus a}) \leq \lceil b/a \rceil^s a$, where $f^s$ denotes taking the pointwise $s$th power.  Therefore, by \cref{symmetrize},
\begin{align*}
\tau(\tens(G)) &\leq \log_{N^s} \rank(\tens_{f^s}(G))\\
&\leq \log_{N^s} \rank((\tens_{f^s}(G))^{\oplus a})\\
&\leq \log_{N^s} \lceil b/a \rceil^s a\\
&= \frac{ s\log_2\lceil \tfrac{b}{a}\rceil + \log_2(a)}{s\log_2(N)},
\end{align*}
which goes to $ \log_N \lceil b/a \rceil$ when $s$ goes to infinity.
\end{proof}

\begin{proof}[\upshape\textbf{Proof of \cref{schonhage}}]
Suppose $r=\borderrank(\bigoplus_{i=1}^p \tens_{f_i}(G))$. This implies that there is an $h\in \NN$ such that $\rank_h(\bigoplus_{i=1}^p \tens_{f_i}(G)) = r$. Taking the $s$th power gives $\rank_{hs}\bigl( (\bigoplus_{i=1}^p \tens_{f_i}(G))^{\otimes s} \bigr) \leq r^s$. We expand the tensor power to get
\[
\rank_{hs}\Bigl( \bigoplus_{\sigma} \Bigl( \tens_{f_1^{\sigma_1}}(G) \otimes \cdots \otimes \tens_{f_p^{\sigma_p}}(G) \Bigr)^{\!\oplus \binom{s}{\sigma}} \Bigr) \leq r^s,
\]
where the first direct sum is over all $p$-tuples $\sigma = (\sigma_1, \ldots, \sigma_p)$ of nonnegative integers with sum $s$ (this is by the multinomial theorem). We can also write this inequality as
\[
\rank_{hs} \Bigl( \bigoplus_{\sigma} \bigl(\tens_{f_1^{\sigma_1}\cdots f_p^{\sigma_p}}(G)\bigr)^{\!\oplus \binom{s}{\sigma}} \Bigr) \leq r^s.
\]
By \cref{borderranktorank} there exists a number $c_{hs}$ which is at most a polynomial in $h$ and $s$ such that
\begin{equation}\label{ub1}
\rank \Bigl( \bigoplus_{\sigma} \bigl(\tens_{f_1^{\sigma_1}\cdots f_p^{\sigma_p}}(G)\bigr)^{\oplus \binom{s}{\sigma}}   \Bigr) \leq c_{hs} r^s.
\end{equation}
Define $\tau$ by $\sum_{i=1}^p \bigl( \prod_{e\in E} f_i(e)  \bigr)^\tau = r$. Then
\begin{equation}\label{eqmultinom}
\sum_\sigma \binom{s}{\sigma}\Bigl( \prod_{e\in E}\prod_{i\in [p]} f_i(e)^{\sigma_i} \Bigr)^{\tau} = r^s,
\end{equation}
by the multinomial theorem again. In this sum, consider the maximum summand and fix $\sigma$ to be the corresponding~$\sigma$ for the remainder of the proof. 

Define $f:E\to\NN$ by $f(e)\coloneqq \prod_i f_i(e)^{\sigma_i}$. Let $a\coloneqq \binom{s}{\sigma}$ and let $b\coloneqq c_{hs} r^s$. Equation \eqref{ub1} implies
$\rank(\tens_f(G)^{\oplus a}) \leq b$ which by  \cref{nonuniformsymm} implies
\begin{equation}\label{taubound}
\tau(\tens(G)) \leq \log_N \bigl\lceil \rank(\tens_f(G)^{\oplus a})/a \bigr\rceil \leq \log_{N} \Bigl\lceil \frac{b}{a} \Bigr\rceil,
\end{equation}
where $N \coloneqq \prod_e f(e)$.
There are $\binom{s+p-1}{p-1} \leq (s+1)^{p-1}$ $p$-tuples $\sigma$ with sum~$s$, so there are that many summands in \eqref{eqmultinom}. We thus lower bound the maximum summand by the average of the summands as follows,
\begin{equation}\label{avgprinciple}
a N^{\tau} \geq \frac{r^s}{\binom{s+p-1}{p-1}} \geq \frac{r^s}{(s+1)^{p-1}}.
\end{equation}
Manipulating \eqref{avgprinciple} gives
\[
\Bigl\lceil \frac{b}{a} \Bigr\rceil \leq \frac{c_{hs} r^s}{a} + 1 \leq N^{\tau} (s+1)^{p-1} 2c_{hs}
\]
which we plug into \eqref{taubound} to get
\begin{align*}
\tau(\tens(G)) &\leq\frac{\tau \log_2 N + (p-1)\log_2(s+1) + \log_2(2c_{hs})}{\log_2 N},\\
&\leq \tau + \frac{(p-1)\log_2(s+1) + \log_2(2c_{hs})}{\log_2 N},
\end{align*}
which goes to $\tau$ when $s$ goes to infinity, because \eqref{avgprinciple} implies
\[
N^\tau \geq \frac{r^s}{(s+1)^{p-1}a} \geq \frac{r^s}{(s+1)^{p-1}p^s}
\]
and therefore $\log_2 N$ is at least $s (\log_2 r - \log_2 p) - (p-1)\log_2(s+1)$, and we assumed $r > p$.
\end{proof}

We are now ready to discuss the final ingredient for the proof of \cref{cwcomplete}, which is a generalization of block decompositions of matrices to block decompositions of tensors.

\begin{definition}[Set partition and tensor partition]
Let $A$ be a finite set. We define a \defin{partition} of $A$ as a collection $\mathcal{A}=\{A^j\}$ of disjoint subsets of~$A$ whose union is $A$.
Let $B_1, \ldots, B_k$ be finite sets, and let $\mathcal{B}_1,\ldots, \mathcal{B}_k$ be partitions of each set respectively. Then we say $\mathcal{B} = \mathcal{B}_1\times \cdots \times \mathcal{B}_k$ is a \defin{product partition} of $B_1\times \cdots \times B_k$.

Let $\phi \in V_1\otimes \cdots \otimes V_k$ be a tensor and $B_1, \ldots, B_k$ orthogonal bases for $V_1, \ldots, V_k$ respectively. Let $\mathcal{B}$ be a product partition of $B_1\times \cdots \times B_k$. Let $I \in \mathcal{B}$. Then we define $\phi_I$ as the orthogonal projection of $\phi$ onto the linear space spanned by~$I$. These smaller tensors $\phi_I$ with $I\in \mathcal{B}$ we think of as making up the \defin{inner structure of~$\phi$ with respect to $\mathcal{B}$}. We define the \defin{outer structure of~$\phi$ with respect to $\mathcal{B}$} to be the tensor~$\phi_{\mathcal{B}}$ with entries indexed by $I\in \mathcal{B}$ such that $\phi_\mathcal{B}$ has a 1 at position $I$ if $\phi_I$ is not the zero tensor, and a 0 otherwise. 
\end{definition}

\begin{definition}
Let $\mathcal{B}$ be a product partition of $B_1\times \cdots \times B_k$ and let $\mathcal{C}$ be a product partition of $C_1 \times \cdots \times C_k$.
Define the product partition $\mathcal{B}\otimes \mathcal{C}$ of $(B_1\times C_1)\times \cdots \times (B_k\times C_k)$ by 
\[
\mathcal{B} \otimes \mathcal{C} \coloneqq \{I \otimes J \mid I\in \mathcal{B}, J\in \mathcal{C}\},
\]
where $I\otimes J = \{((b_1, c_1), \ldots, (b_k, c_k))  \mid b\in I, c\in J\}$.
\end{definition}

The following proposition follows directly from the definition.

\begin{proposition}\label{partitionproduct}
Let $\phi\in U_1\otimes \cdots \otimes U_k$ and $\psi \in V_1\otimes \cdots \otimes V_k$. Let $B_i$ be a basis of $U_i$ and let $C_i$ be basis of $V_i$. Let $\mathcal{B}$ be a product partition of $B_1 \times \cdots \times B_k$ and let $\mathcal{C}$ be a product partition of $C_1\times \cdots \times C_k$. Then,
\[
(\phi \otimes \psi)_{\mathcal{B} \otimes \mathcal{C}} = \phi_{\mathcal{B}} \otimes \psi_{\mathcal{C}},
\]
and $(\phi\otimes \psi)_{I\otimes J} = \phi_I \otimes \psi_J$ for all $I\in \mathcal{B}$, $J\in \mathcal{C}$.
\end{proposition}

\subsection{Proof of \cref{cwcomplete}}
We define a modification of the ``Coppersmith--Winograd tensor'', which is a higher-dimensional version of the weight-$(2,k-2)$ Dicke tensor. 

\begin{definition} Let $k,q\geq 2$ be integers. Let $b_0, b_1, \ldots, b_q$ be the standard basis of $\CC^{q+1}$.
Let $\CW_q^k$ be the $k$-tensor
\[
\CW_q^k \coloneqq \sum_{\mathclap{i \in \{0,1,\ldots,q\}^k}} b_{i_1} \otimes \cdots \otimes b_{i_k} \in (\CC^{q+1})^{\otimes k}
\]
where the sum goes over all $i$ that contain exactly two identical non-zero entries. One can also think of this tensor as a sum of tensors $\tens_q(e)$ with $e$ going over the edges of the complete graph $K_k$, that is, using subscripts to denote the position of tensor legs,
\begin{equation}\label{defcw}
\CW_q^k =\!\!\!\!\!\sum_{e \in E(K_k)} \sum_{i=1}^q  (b_i \otimes b_i)_e \otimes (b_0\otimes b_0\otimes \cdots \otimes b_0)_{[k]\setminus e} \in (\CC^{q+1})^{\otimes k}.
\end{equation}
Note that in \eqref{defcw} the sums runs over $i$ from 1 to $q$.
\end{definition}

The $\CW_q^k$-tensor should be thought of as a cheap tensor, in the following border rank sense.

\begin{lemma}\label{CWborder}
The border rank of $\CW_q^k$ is at most $q+2$.
\end{lemma}
\begin{proof}
We have
\begin{align*}
& \sum_{i=1}^q \varepsilon\,(b_0 + \varepsilon^2\, b_i)^{\otimes k} - \Bigl(b_0 + \varepsilon^3 \sum_{i=1}^q b_i\Bigr)^{\otimes k} + (1 - q\varepsilon) b_0^{\otimes k}\\
&\quad= \varepsilon^5 \CW_q^k + \mathcal{O}(\varepsilon^6).
\end{align*}
Therefore, $\borderrank(\CW_q^k) \leq q + 2$.
\end{proof}

\begin{customtheorem}{\ref{cwcomplete}}[Repeated]
Let $k\geq 2$.  For any $q\geq 1$, 
\[
\tau(\tens(K_k)) \leq \log_q \Bigl(\displaystyle \frac{q + 2}{2^{\combsubomega(D_{(k-2,2)})}} \Bigr).
\]
\end{customtheorem}

\begin{proof}[\upshape\textbf{Proof of \cref{cwcomplete}}]
Define a product partition $\mathcal{B} = \mathcal{B}_1\times \cdots \times \mathcal{B}_k$ of the standard basis of $(\CC^{q+1})^{\otimes k}$ by $\mathcal{B}_j \coloneqq \{\{b_0\}, \{b_1,\ldots, b_q\}\}$ for all $j$. Then the outer structure $(\CW_q^k)_{\mathcal{B}}$ equals
\[
(\CW_q^k)_{\mathcal{B}} = \sum_{\mathclap{e\in E(K_k)}} (b_1\otimes b_1)_e \otimes (b_0 \otimes b_0 \otimes \cdots \otimes b_0)_{[k]\setminus e}
\]
which is the Dicke tensor $D_{(k-2, 2)}$. (We implicitly reorder $(k-2,2)$ to $(2,k-2)$ when $k\leq 3$.)
The inner structure of $\CW_q^k$ with respect to $\mathcal{B}$ consists of tensors
\[
\sum_{i=1}^q (b_i\otimes b_i)_e \otimes (b_0 \otimes b_0 \otimes \cdots \otimes b_0)_{[k]\setminus e} \quad\textnormal{with $e\in E(K_k)$}.
\]
By \cref{charac2}, we have the following monomial degeneration on the outer structure level: %
\begin{equation}\label{combeq1}
\tens^{\otimes(\combsubomega(D_{(k-2,2)})s - o(s))} \combdegenleq ((\CW_q^k)^{\otimes s})_{\mathcal{B}^{\otimes s}}.
\end{equation}
Since the degeneration in \eqref{combeq1} is monomial, each nonzero entry in the lowest-degree part of the tensor on the left-hand side corresponds to an inner structure tensor of $(\CW_q^k)^{\otimes s}$ with respect to $\mathcal{B}^{\otimes s}$.  Each such inner structure tensor is a product of $s$ inner structure tensors of $\CW_q^k$ with respect to $\mathcal{B}$ (\cref{partitionproduct}). Therefore,
\[
\phi_1 \oplus \cdots \oplus \phi_p \degenleq (\CW_q^k)^{\otimes s}
\]
with $p=2^{\combsubomega(D_{(k-2,2)})s - o(s)}$ and $\phi_i \in \{\tens_f(K_k)\mid \prod_{e\in E} f(e) = q^s\}$.
Then, by \cref{CWborder} we have
\[
\borderrank\bigl( \phi_1 \oplus \cdots \oplus \phi_p \bigr) \leq \borderrank((\CW_q^k)^{\otimes s}) \leq (q+2)^s 
\]
We apply the ``in particular'' case of the generalized asymptotic sum inequality \cref{schonhage} with $r$ defined as $(q+2)^s$ %
to the graph $G = K_k$
to obtain
\begin{align*}
\tau(\tens(K_k)) &\leq \log_{q^s} \frac{(q+2)^s}{2^{ \combsubomega(D_{(k-2,2)})s - o(s)}}\\
&\leq \log_{q} \frac{q+2}{2^{ \combsubomega(D_{(k-2,2)}) - o(1)}}.
\end{align*}
Letting $s$ go to infinity yields the required inequality.
\end{proof}

\paragraph{Summary.} We wish to give a short summary of the generalized Coppersmith--Winograd method as employed above, for $k=4$. Let $K_4=(E,V)$ be the tetrahedron. We aim to get a good upper bound on the tensor rank of a tensor power of the goal tensor:
\newcommand{\mahis}{\hspace{1.1em}=\hspace{1.1em}}
\newcommand{\mahisplus}{\hspace{1.1em}+\hspace{1.1em}}
\newcommand{\mahplus}{\hspace{-0.1em}+}
\newcommand{\mahcomma}{\hspace{-0.1ex}\raisebox{-0.2em}{,}}
\def\grsize{1.2em}
\begin{alignat*}{3}
\textsf{goal tensor} &\mahis  \gtens{\pbox{2em}{
\begin{tikzpicture}[line width=0.2mm, vertex/.style={
    circle,
    fill=white,
    draw,
    outer sep=0pt,
    inner sep=0.1em}, x=\grsize, y=\grsize]
    \path[coordinate] (0,0)  coordinate(A)
                ( 0,1) coordinate(B)
                ( 1,1) coordinate(C)
                ( 1,0) coordinate(D);
	\draw [line width=0.2mm] (A) -- (C);
	\draw [line width=0.2mm] (B) -- (D);
	\draw [line width=0.2mm]  (A) node [vertex] {} -- (B) node [vertex] {} -- (C) node [vertex] {} -- (D) node [vertex] {} -- (A);
\end{tikzpicture}
}}\\[0.5em]
\qquad
&\mahis \bigotimes_{e\in E} \sum_{i\in [2]} (b_i \otimes b_i)_e \otimes (b_0 \otimes b_0)_{V\setminus\{e\}}.\\
\intertext{To this end, we pick the following starting tensor, for which we have a good border rank upper bound (\cref{CWborder}):}
\textsf{starting tensor} &\mahis
 \gtensq{\pbox{1.9em}{
\begin{tikzpicture}[line width=0.2mm, vertex/.style={
    circle,
    fill=white,
    draw,
    outer sep=0pt,
    inner sep=0.1em}, x=\grsize, y=\grsize]
    \path[coordinate] (0,0)  coordinate(A)
                ( 0,1) coordinate(B)
                ( 1,1) coordinate(C)
                ( 1,0) coordinate(D);
	\draw (A) node [vertex] {};
	\draw (B) node [vertex] {};
	\draw (C) node [vertex] {};
	\draw (D) node [vertex] {};
	\draw [line width=0.2mm]  (A) node [vertex] {} -- (B) node [vertex] {};
\end{tikzpicture}
}}
\mahplus
\gtensq{\pbox{1.9em}{
\begin{tikzpicture}[line width=0.2mm, vertex/.style={
    circle,
    fill=white,
    draw,
    outer sep=0pt,
    inner sep=0.1em}, x=\grsize, y=\grsize]
    \path[coordinate] (0,0)  coordinate(A)
                ( 0,1) coordinate(B)
                ( 1,1) coordinate(C)
                ( 1,0) coordinate(D);
	\draw (A) node [vertex] {};
	\draw (B) node [vertex] {};
	\draw (C) node [vertex] {};
	\draw (D) node [vertex] {};
	\draw [line width=0.2mm]  (A) node [vertex] {} -- (C) node [vertex] {};
\end{tikzpicture}
}}
\mahplus
\gtensq{\pbox{1.9em}{
\begin{tikzpicture}[line width=0.2mm, vertex/.style={
    circle,
    fill=white,
    draw,
    outer sep=0pt,
    inner sep=0.1em}, x=\grsize, y=\grsize]
    \path[coordinate] (0,0)  coordinate(A)
                ( 0,1) coordinate(B)
                ( 1,1) coordinate(C)
                ( 1,0) coordinate(D);
	\draw (A) node [vertex] {};
	\draw (B) node [vertex] {};
	\draw (C) node [vertex] {};
	\draw (D) node [vertex] {};
	\draw [line width=0.2mm]  (A) node [vertex] {} -- (D) node [vertex] {};
\end{tikzpicture}
}}\\[0.5em] 
&{}\hspace{1.1em}{}+{}\hspace{1.1em}
\gtensq{\pbox{1.9em}{
\begin{tikzpicture}[line width=0.2mm, vertex/.style={
    circle,
    fill=white,
    draw,
    outer sep=0pt,
    inner sep=0.1em}, x=\grsize, y=\grsize]
    \path[coordinate] (0,0)  coordinate(A)
                ( 0,1) coordinate(B)
                ( 1,1) coordinate(C)
                ( 1,0) coordinate(D);
	\draw (A) node [vertex] {};
	\draw (B) node [vertex] {};
	\draw (C) node [vertex] {};
	\draw (D) node [vertex] {};
	\draw [line width=0.2mm]  (B) node [vertex] {} -- (C) node [vertex] {};
\end{tikzpicture}
}}
\mahplus
\gtensq{\pbox{1.9em}{
\begin{tikzpicture}[line width=0.2mm, vertex/.style={
    circle,
    fill=white,
    draw,
    outer sep=0pt,
    inner sep=0.1em}, x=\grsize, y=\grsize]
    \path[coordinate] (0,0)  coordinate(A)
                ( 0,1) coordinate(B)
                ( 1,1) coordinate(C)
                ( 1,0) coordinate(D);
	\draw (A) node [vertex] {};
	\draw (B) node [vertex] {};
	\draw (C) node [vertex] {};
	\draw (D) node [vertex] {};
	\draw [line width=0.2mm]  (B) node [vertex] {} -- (D) node [vertex] {};
\end{tikzpicture}
}}
\mahplus
\gtensq{\pbox{1.9em}{
\begin{tikzpicture}[line width=0.2mm, vertex/.style={
    circle,
    fill=white,
    draw,
    outer sep=0pt,
    inner sep=0.1em}, x=\grsize, y=\grsize]
    \path[coordinate] (0,0)  coordinate(A)
                ( 0,1) coordinate(B)
                ( 1,1) coordinate(C)
                ( 1,0) coordinate(D);
	\draw (A) node [vertex] {};
	\draw (B) node [vertex] {};
	\draw (C) node [vertex] {};
	\draw (D) node [vertex] {};
	\draw [line width=0.2mm]  (C) node [vertex] {} -- (D) node [vertex] {};
\end{tikzpicture}
}}\\[0.5em]
&\mahis\sum_{e\in E} \sum_{i\in [q]} (b_i \otimes b_i)_e \otimes (b_0 \otimes b_0)_{V\setminus\{e\}}.\\
\intertext{We now choose a product partition of the standard basis for the space that the starting tensor lives in. This partitions the starting tensors into blocks (like we can partition matrices into blocks). The partition we use is $\{\{b_0\}, \{b_1, \ldots, b_q\}\}^{\times 4}$. The partitioned tensor has, besides the zero block, the following blocks:}
\textsf{inner structure} &\mahis \Bigl\{{}\gtensq{\pbox{1.9em}{
\begin{tikzpicture}[line width=0.2mm, vertex/.style={
    circle,
    fill=white,
    draw,
    outer sep=0pt,
    inner sep=0.1em}, x=\grsize, y=\grsize]
    \path[coordinate] (0,0)  coordinate(A)
                ( 0,1) coordinate(B)
                ( 1,1) coordinate(C)
                ( 1,0) coordinate(D);
	\draw (A) node [vertex] {};
	\draw (B) node [vertex] {};
	\draw (C) node [vertex] {};
	\draw (D) node [vertex] {};
	\draw [line width=0.2mm]  (A) node [vertex] {} -- (B) node [vertex] {};
\end{tikzpicture}
}}
\mahcomma
\gtensq{\pbox{1.9em}{
\begin{tikzpicture}[line width=0.2mm, vertex/.style={
    circle,
    fill=white,
    draw,
    outer sep=0pt,
    inner sep=0.1em}, x=\grsize, y=\grsize]
    \path[coordinate] (0,0)  coordinate(A)
                ( 0,1) coordinate(B)
                ( 1,1) coordinate(C)
                ( 1,0) coordinate(D);
	\draw (A) node [vertex] {};
	\draw (B) node [vertex] {};
	\draw (C) node [vertex] {};
	\draw (D) node [vertex] {};
	\draw [line width=0.2mm]  (A) node [vertex] {} -- (C) node [vertex] {};
\end{tikzpicture}
}}
\mahcomma
\gtensq{\pbox{1.9em}{
\begin{tikzpicture}[line width=0.2mm, vertex/.style={
    circle,
    fill=white,
    draw,
    outer sep=0pt,
    inner sep=0.1em}, x=\grsize, y=\grsize]
    \path[coordinate] (0,0)  coordinate(A)
                ( 0,1) coordinate(B)
                ( 1,1) coordinate(C)
                ( 1,0) coordinate(D);
	\draw (A) node [vertex] {};
	\draw (B) node [vertex] {};
	\draw (C) node [vertex] {};
	\draw (D) node [vertex] {};
	\draw [line width=0.2mm]  (A) node [vertex] {} -- (D) node [vertex] {};
\end{tikzpicture}
}}
\mahcomma\\[0.5em]
&\hspace{1.1em}\phantom{{}={}}\hspace{1.1em}\phantom{\Bigl\{}\gtensq{\pbox{1.9em}{
\begin{tikzpicture}[line width=0.2mm, vertex/.style={
    circle,
    fill=white,
    draw,
    outer sep=0pt,
    inner sep=0.1em}, x=\grsize, y=\grsize]
    \path[coordinate] (0,0)  coordinate(A)
                ( 0,1) coordinate(B)
                ( 1,1) coordinate(C)
                ( 1,0) coordinate(D);
	\draw (A) node [vertex] {};
	\draw (B) node [vertex] {};
	\draw (C) node [vertex] {};
	\draw (D) node [vertex] {};
	\draw [line width=0.2mm]  (B) node [vertex] {} -- (C) node [vertex] {};
\end{tikzpicture}
}}
\mahcomma
\gtensq{\pbox{1.9em}{
\begin{tikzpicture}[line width=0.2mm, vertex/.style={
    circle,
    fill=white,
    draw,
    outer sep=0pt,
    inner sep=0.1em}, x=\grsize, y=\grsize]
    \path[coordinate] (0,0)  coordinate(A)
                ( 0,1) coordinate(B)
                ( 1,1) coordinate(C)
                ( 1,0) coordinate(D);
	\draw (A) node [vertex] {};
	\draw (B) node [vertex] {};
	\draw (C) node [vertex] {};
	\draw (D) node [vertex] {};
	\draw [line width=0.2mm]  (B) node [vertex] {} -- (D) node [vertex] {};
\end{tikzpicture}
}}
\mahcomma
\gtensq{\pbox{1.9em}{
\begin{tikzpicture}[line width=0.2mm, vertex/.style={
    circle,
    fill=white,
    draw,
    outer sep=0pt,
    inner sep=0.1em}, x=\grsize, y=\grsize]
    \path[coordinate] (0,0)  coordinate(A)
                ( 0,1) coordinate(B)
                ( 1,1) coordinate(C)
                ( 1,0) coordinate(D);
	\draw (A) node [vertex] {};
	\draw (B) node [vertex] {};
	\draw (C) node [vertex] {};
	\draw (D) node [vertex] {};
	\draw [line width=0.2mm]  (C) node [vertex] {} -- (D) node [vertex] {};
\end{tikzpicture}
}}\Bigr\}\\[0.5em]
&\mahis \Bigl\{ \sum_{i\in [q]} (b_i \otimes b_i)_e \otimes (b_0 \otimes b_0)_{V\setminus\{e\}} \,\,\Big|\,\, e \in E \Bigr\}.\\
\intertext{The following outer structure tensor tells us how the nonzero blocks are positioned, in block coordinates:}
\textsf{outer structure}  &\mahis D_{(2,2)}\\
&\mahis \sum_{e\in E} (b_1 \otimes b_1)_e \otimes (b_0 \otimes b_0)_{V\setminus\{e\}}.
\end{alignat*}
We can degenerate $D_{(2,2)}^{\otimes s}$ to $\sum_{i\in [p]} b_i \otimes b_i \otimes b_i \otimes b_i$ for  $p=2^{\combsubomega(D_{(2,2)})s - o(s)}$ by a monomial degeneration (\cref{defcombsubexp}). This implies that we can degenerate the $s$th power of the starting tensor to a direct sum of nonuniform goal tensors, but such that for each nonuniform goal tensor the product of the edge weights equals $q^s$ (by \cref{partitionproduct}):
\begin{alignat*}{3}
&\gtensf{\pbox{2em}{
\begin{tikzpicture}[line width=0.2mm, vertex/.style={
    circle,
    fill=white,
    draw,
    outer sep=0pt,
    inner sep=0.1em}, x=\grsize, y=\grsize]
    \path[coordinate] (0,0)  coordinate(A)
                ( 0,1) coordinate(B)
                ( 1,1) coordinate(C)
                ( 1,0) coordinate(D);
	\draw [line width=0.2mm] (A) -- (C);
	\draw [line width=0.2mm] (B) -- (D);
	\draw [line width=0.2mm]  (A) node [vertex] {} -- (B) node [vertex] {} -- (C) node [vertex] {} -- (D) node [vertex] {} -- (A);
\end{tikzpicture}
}}{f_1}
\oplus
\gtensf{\pbox{2em}{
\begin{tikzpicture}[line width=0.2mm, vertex/.style={
    circle,
    fill=white,
    draw,
    outer sep=0pt,
    inner sep=0.1em}, x=\grsize, y=\grsize]
    \path[coordinate] (0,0)  coordinate(A)
                ( 0,1) coordinate(B)
                ( 1,1) coordinate(C)
                ( 1,0) coordinate(D);
	\draw [line width=0.2mm] (A) -- (C);
	\draw [line width=0.2mm] (B) -- (D);
	\draw [line width=0.2mm]  (A) node [vertex] {} -- (B) node [vertex] {} -- (C) node [vertex] {} -- (D) node [vertex] {} -- (A);
\end{tikzpicture}
}}{f_2} \oplus \cdots.
\qquad 
\end{alignat*}
We then apply the generalized asymptotic sum inequality \cref{schonhage}, whose construction consists of the following two steps. We first restrict a power of the above tensor to a direct power of a single nonuniform goal tensor
\begin{alignat*}{3}
&\gtensf{\pbox{2em}{
\begin{tikzpicture}[line width=0.2mm, vertex/.style={
    circle,
    fill=white,
    draw,
    outer sep=0pt,
    inner sep=0.1em}, x=\grsize, y=\grsize]
    \path[coordinate] (0,0)  coordinate(A)
                ( 0,1) coordinate(B)
                ( 1,1) coordinate(C)
                ( 1,0) coordinate(D);
	\draw [line width=0.2mm] (A) -- (C);
	\draw [line width=0.2mm] (B) -- (D);
	\draw [line width=0.2mm]  (A) node [vertex] {} -- (B) node [vertex] {} -- (C) node [vertex] {} -- (D) node [vertex] {} -- (A);
\end{tikzpicture}
}}{f}^{\smash{\!\oplus a}}.
\qquad 
\end{alignat*}
Second, another tensor power and symmetrization procedure gives a power of the goal tensor:
\begin{alignat*}{3}
&\gtens{\pbox{1.9em}{
\begin{tikzpicture}[line width=0.2mm, vertex/.style={
    circle,
    fill=white,
    draw,
    outer sep=0pt,
    inner sep=0.1em}, x=\grsize, y=\grsize]
    \path[coordinate] (0,0)  coordinate(A)
                ( 0,1) coordinate(B)
                ( 1,1) coordinate(C)
                ( 1,0) coordinate(D);
	\draw [line width=0.2mm] (A) -- (C);
	\draw [line width=0.2mm] (B) -- (D);
	\draw [line width=0.2mm]  (A) node [vertex] {} -- (B) node [vertex] {} -- (C) node [vertex] {} -- (D) node [vertex] {} -- (A);
\end{tikzpicture}
}}^{\smash{\!\otimes N}}.
\qquad 
\end{alignat*}

\section{Lower bound on the monomial subexponent of tight tensors}\label{secmain}

We first discuss some results on $k$-average-free sets, linear combinations of independent uniformly random variables and types of sequences, that we will use in the proof of \cref{main}. Then we describe a basic procedure to restrict any tensor to a tensor of the form~$\tens_r$ with a monomial restriction. Then, we give the proof for \cref{main}. Next we discuss some computational aspects of \cref{main}. Finally, we give some applications.

\subsection{Preliminaries}

We begin with a result on $k$-average-free sets which is essentially due to Salem and Spencer \cite{salem1942sets}. For a self-contained proof, see \cite[Lemma~10]{vrana2015asymptotic}.

\begin{definition}
A subset $A\subseteq \NN$ is \defin{$k$-average-free} if $a_1,\ldots, a_k,y\in A$ and $a_1 + \cdots + a_k = ky$ implies $a_1 = \cdots = a_k = y$. %
For $N\in \NN$, let $\nu_k(N)$ be the maximum size of a $k$-average-free set in $[N] = \{1,2, \ldots, N\}$.
\end{definition}

\begin{proposition}\label{avgfree}
For any fixed $k$, we have $\nu_k(N) = N^{1-o(1)}$ as $N \to \infty$.
\end{proposition}

The following statement about linear combinations of independent uniformly random variables is standard.

\begin{proposition}\label{uniflem} Let $M$ be a prime.
Let $u_1, \ldots, u_n$ be independent uniformly distributed random variables in $\ZZ/M\ZZ$. Let $b_1, \ldots, b_m$ be $(\ZZ/M\ZZ)$-linear combinations of $u_1,\ldots,u_n$.
Then the vector $b=(b_1, \ldots, b_m)$ is uniformly randomly distributed on the range of $b$ in $(\ZZ/M\ZZ)^m$.
\end{proposition}
\begin{proof}
Suppose $b_i = \sum_j c_{ij} u_j$. So $b = C u$ where $u = (u_1,\ldots, u_n)$ and $C$ is the matrix with entries $C_{ij} = c_{ij}$. For any $y$ in the image of $C$, the cardinality of the preimage $C^{-1}(y)$ is exactly the cardinality of the kernel of $C$. Indeed, if $Cx = y$, then $C^{-1}(y) = x + \ker(C)$. Since $u$ is uniform, we conclude that $b$ is uniform on the image of $C$.
\end{proof}

The method of types classifies sequences of symbols according to the relative proportion of occurrences of each symbol. We have used this method before in the proof of the generalized asymptotic sum inequality. In the proof of \cref{main} it will play a more important role. %

\begin{definition}
Let $N\in \NN$ and let $X$ be a finite ``alphabet'' set. The \defin{type} $P_x$ of a sequence $x=(x_1,x_2,\ldots, x_N)\in X^N$ is the relative proportion of occurrences of each symbol of $X$.  Type is sometimes called empirical distribution. The possible types for sequences $x\in X^N$ are called the \defin{$N$-types on $X$}.  If $P$ is an $N$-type on $X$, then the set of sequences $x\in X^N$ of type $P$ is called the \defin{type class} of $P$ and is denoted by $T_P^N$. 
\end{definition}
\begin{example}
For $X = \{0,1\}$ the type of the sequence $x=(1,0,0)$ is given by $P_x(0) = 2/3$ and $P_x(1) = 1/3$. The possible 3-types are given by $P(0) = 0$, $P(0) = 1/3$, $P(0) = 2/3$, $P(0) = 1$ (with $P(1) = 1-P(0)$). The type class of the 3-type given by $P(0) = 2/3$ is $T_P^3=\{(1,0,0), (0,1,0), (0,0,1)\}$.
\end{example}

\begin{proposition}\label{types}
Let $N\in \NN$ and let $X$ be a finite set.
The number of $N$-type classes on $X$ is the binomial coefficient $\binom{N+|X|-1}{|X|-1} = 2^{o(N)}$.
Let $P$ be an $N$-type on $X$. The number of sequences in the type class of $P = (p_1, p_2, \ldots)$ is the multinomial coefficient \[\binom{N}{Np_1,Np_2, \ldots},\] which is lower bounded by $2^{NH(P)-o(N)}$ and upper bounded by $2^{NH(P)}$.
\end{proposition}
\begin{proof}
The first statement can be proved with the famous stars and bars argument: count the number of ways to arrange $N$ stars and $|X|-1$ bars in a row. For a proof of the second statement see \cite[Theorem 12.1.3]{cover2012elements}.
\end{proof}

\subsection{Restriction procedure}

We describe a procedure for finding a monomial restriction (see \cref{restrict} for the definition of monomial restriction) from any tensor $\psi$ to~$\tens_r$ for some~$r$. Despite its simplicity, this algorithm lies at the heart of the proof of \cref{main}. Typically, we will apply the algorithm to a tensor~$\psi$ of the form $\phi^{\otimes N}$ since we care about the monomial subexponent.

In a graph, a connected component is a maximal connected subgraph. Each vertex of a graph belongs to exactly one connected component.

\begin{lemma}\label{concomp}
Let $G$ be a graph with $n$ vertices and $m$ edges. Then $G$ has at least $n-m$ connected components.
\end{lemma}
\begin{proof}
A graph without edges has $n$ connected components. For every edge that we add to the graph, we lose at most one connected component.
\end{proof}

\begin{proposition}\label{elimprop}
Let $\psi$ be a tensor in $V_1 \otimes \cdots \otimes V_k$ and let $B_1, \ldots, B_k$ be bases for $V_1, \ldots, V_k$. Let $\Psi$ be the corresponding support of $\psi$. Let $\badpairs$ be the set $\{\{a,a'\} \subseteq \Psi \mid a\neq a';\,  \exists i\, a_i = a'_i \}$. Then $\psi \combgeq \tens_{|\Psi|-|\badpairs|}$.
\end{proposition}
\begin{proof}
Let $G$ be the graph with vertex set $\Psi$ and edge set $\badpairs$. Let $\mathcal{E}(\Psi)\subseteq \Psi$ be a subset of the vertex set that contains exactly one vertex per connected component of $G$. Then $\abs[0]{\mathcal{E}(\Psi)} \geq |\Psi|-|\badpairs|$ by \cref{concomp}. Moreover, because the vertices in $\mathcal{E}(\Psi)$ are non-adjacent, there is a monomial restriction $\psi \combgeq \tens_{|\Psi|-|\badpairs|}$.
\end{proof}

We illustrate the procedure in \cref{elimprop} with a concrete example where $\psi$ is of the form $\phi^{\otimes 2}$.

\begin{example}\label{algill}
Let $D_k = D_{(1^k)}$ be the $k$-tensor $\sum_{\sigma \in S_k} \sigma \cdot b_1\otimes b_2\otimes \cdots\otimes b_k$ living in $(\CC^k)^{\otimes k}$, where the symmetric group $S_k$ acts by permuting tensor legs. In the standard basis, and identifying $b_i$ with $i$, the tensor $D_{k}$ has support
\[
\Phi = \{(\sigma(1),\sigma(2),\ldots, \sigma(k)) \mid \sigma \in S_k\} \subseteq [k]^k.
\]
Consider the $N$th power $D_k^{\otimes N} \in (\CC^{k^N})^{\otimes k}$. Let $\Phi^{\otimes N} \subseteq ([k]^N)^k$ be its support in the tensor product basis. We write every element in $\Phi^{\otimes N}$ as a $k$-tuple $(I_1, \ldots, I_k)$ of column $N$-vectors~$I_i$, so that for $k=3$ and $N=2$ we get
\newcommand{\lowcomma}{\raisebox{-0.3em}{\hspace{-0.05em},\hspace{0.1em}}}
\newcommand{\elowcomma}{\!\lowcomma}
\begin{align*}
\Phi^{\otimes 2} = \bigl\{&\bigl(\begin{bsmallmatrix}0\\ 0\end{bsmallmatrix}\lowcomma \begin{bsmallmatrix}1\\ 1\end{bsmallmatrix}\lowcomma \begin{bsmallmatrix}2\\ 2\end{bsmallmatrix}\bigl)\elowcomma\,
 \bigl(\begin{bsmallmatrix}0\\ 0\end{bsmallmatrix}\lowcomma \begin{bsmallmatrix}1\\ 2\end{bsmallmatrix}\lowcomma \begin{bsmallmatrix}2\\ 1\end{bsmallmatrix}\bigl)\elowcomma\,
 \bigl(\begin{bsmallmatrix}0\\ 1\end{bsmallmatrix}\lowcomma \begin{bsmallmatrix}1\\ 0\end{bsmallmatrix}\lowcomma \begin{bsmallmatrix}2\\ 2\end{bsmallmatrix}\bigl)\elowcomma\,
 \bigl(\begin{bsmallmatrix}0\\ 1\end{bsmallmatrix}\lowcomma \begin{bsmallmatrix}1\\ 2\end{bsmallmatrix}\lowcomma \begin{bsmallmatrix}2\\ 0\end{bsmallmatrix}\bigl)\elowcomma\,\\
 &\bigl(\begin{bsmallmatrix}0\\ 2\end{bsmallmatrix}\lowcomma \begin{bsmallmatrix}1\\ 0\end{bsmallmatrix}\lowcomma \begin{bsmallmatrix}2\\ 1\end{bsmallmatrix}\bigl)\elowcomma\,
 \bigl(\begin{bsmallmatrix}0\\ 2\end{bsmallmatrix}\lowcomma \begin{bsmallmatrix}1\\ 1\end{bsmallmatrix}\lowcomma \begin{bsmallmatrix}2\\ 0\end{bsmallmatrix}\bigl)\elowcomma\,
 \bigl(\begin{bsmallmatrix}0\\ 0\end{bsmallmatrix}\lowcomma \begin{bsmallmatrix}2\\ 1\end{bsmallmatrix}\lowcomma \begin{bsmallmatrix}1\\ 2\end{bsmallmatrix}\bigl)\elowcomma\,
 \bigl(\begin{bsmallmatrix}0\\ 0\end{bsmallmatrix}\lowcomma \begin{bsmallmatrix}2\\ 2\end{bsmallmatrix}\lowcomma \begin{bsmallmatrix}1\\ 1\end{bsmallmatrix}\bigl)\elowcomma\,\\
 &\bigl(\begin{bsmallmatrix}0\\ 1\end{bsmallmatrix}\lowcomma \begin{bsmallmatrix}2\\ 0\end{bsmallmatrix}\lowcomma \begin{bsmallmatrix}1\\ 2\end{bsmallmatrix}\bigl)\elowcomma\,
 \bigl(\begin{bsmallmatrix}0\\ 1\end{bsmallmatrix}\lowcomma \begin{bsmallmatrix}2\\ 2\end{bsmallmatrix}\lowcomma \begin{bsmallmatrix}1\\ 0\end{bsmallmatrix}\bigl)\elowcomma\,
 \bigl(\begin{bsmallmatrix}0\\ 2\end{bsmallmatrix}\lowcomma \begin{bsmallmatrix}2\\ 0\end{bsmallmatrix}\lowcomma \begin{bsmallmatrix}1\\ 1\end{bsmallmatrix}\bigl)\elowcomma\,
 \bigl(\begin{bsmallmatrix}0\\ 2\end{bsmallmatrix}\lowcomma \begin{bsmallmatrix}2\\ 1\end{bsmallmatrix}\lowcomma \begin{bsmallmatrix}1\\ 0\end{bsmallmatrix}\bigl)\elowcomma\,\\
 &\bigl(\begin{bsmallmatrix}1\\ 0\end{bsmallmatrix}\lowcomma \begin{bsmallmatrix}0\\ 1\end{bsmallmatrix}\lowcomma \begin{bsmallmatrix}2\\ 2\end{bsmallmatrix}\bigl)\elowcomma\,
 \bigl(\begin{bsmallmatrix}1\\ 0\end{bsmallmatrix}\lowcomma \begin{bsmallmatrix}0\\ 2\end{bsmallmatrix}\lowcomma \begin{bsmallmatrix}2\\ 1\end{bsmallmatrix}\bigl)\elowcomma\,
 \bigl(\begin{bsmallmatrix}1\\ 1\end{bsmallmatrix}\lowcomma \begin{bsmallmatrix}0\\ 0\end{bsmallmatrix}\lowcomma \begin{bsmallmatrix}2\\ 2\end{bsmallmatrix}\bigl)\elowcomma\,
 \bigl(\begin{bsmallmatrix}1\\ 1\end{bsmallmatrix}\lowcomma \begin{bsmallmatrix}0\\ 2\end{bsmallmatrix}\lowcomma \begin{bsmallmatrix}2\\ 0\end{bsmallmatrix}\bigl)\elowcomma\,\\
 &\bigl(\begin{bsmallmatrix}1\\ 2\end{bsmallmatrix}\lowcomma \begin{bsmallmatrix}0\\ 0\end{bsmallmatrix}\lowcomma \begin{bsmallmatrix}2\\ 1\end{bsmallmatrix}\bigl)\elowcomma\,
 \bigl(\begin{bsmallmatrix}1\\ 2\end{bsmallmatrix}\lowcomma \begin{bsmallmatrix}0\\ 1\end{bsmallmatrix}\lowcomma \begin{bsmallmatrix}2\\ 0\end{bsmallmatrix}\bigl)\elowcomma\,
 \bigl(\begin{bsmallmatrix}1\\ 0\end{bsmallmatrix}\lowcomma \begin{bsmallmatrix}2\\ 1\end{bsmallmatrix}\lowcomma \begin{bsmallmatrix}0\\ 2\end{bsmallmatrix}\bigl)\elowcomma\,
 \bigl(\begin{bsmallmatrix}1\\ 0\end{bsmallmatrix}\lowcomma \begin{bsmallmatrix}2\\ 2\end{bsmallmatrix}\lowcomma \begin{bsmallmatrix}0\\ 1\end{bsmallmatrix}\bigl)\elowcomma\,\\
 &\bigl(\begin{bsmallmatrix}1\\ 1\end{bsmallmatrix}\lowcomma \begin{bsmallmatrix}2\\ 0\end{bsmallmatrix}\lowcomma \begin{bsmallmatrix}0\\ 2\end{bsmallmatrix}\bigl)\elowcomma\,
 \bigl(\begin{bsmallmatrix}1\\ 1\end{bsmallmatrix}\lowcomma \begin{bsmallmatrix}2\\ 2\end{bsmallmatrix}\lowcomma \begin{bsmallmatrix}0\\ 0\end{bsmallmatrix}\bigl)\elowcomma\,
 \bigl(\begin{bsmallmatrix}1\\ 2\end{bsmallmatrix}\lowcomma \begin{bsmallmatrix}2\\ 0\end{bsmallmatrix}\lowcomma \begin{bsmallmatrix}0\\ 1\end{bsmallmatrix}\bigl)\elowcomma\,
 \bigl(\begin{bsmallmatrix}1\\ 2\end{bsmallmatrix}\lowcomma \begin{bsmallmatrix}2\\ 1\end{bsmallmatrix}\lowcomma \begin{bsmallmatrix}0\\ 0\end{bsmallmatrix}\bigl)\elowcomma\,\\
 &\bigl(\begin{bsmallmatrix}2\\ 0\end{bsmallmatrix}\lowcomma \begin{bsmallmatrix}0\\ 1\end{bsmallmatrix}\lowcomma \begin{bsmallmatrix}1\\ 2\end{bsmallmatrix}\bigl)\elowcomma\,
 \bigl(\begin{bsmallmatrix}2\\ 0\end{bsmallmatrix}\lowcomma \begin{bsmallmatrix}0\\ 2\end{bsmallmatrix}\lowcomma \begin{bsmallmatrix}1\\ 1\end{bsmallmatrix}\bigl)\elowcomma\,
 \bigl(\begin{bsmallmatrix}2\\ 1\end{bsmallmatrix}\lowcomma \begin{bsmallmatrix}0\\ 0\end{bsmallmatrix}\lowcomma \begin{bsmallmatrix}1\\ 2\end{bsmallmatrix}\bigl)\elowcomma\,
 \bigl(\begin{bsmallmatrix}2\\ 1\end{bsmallmatrix}\lowcomma \begin{bsmallmatrix}0\\ 2\end{bsmallmatrix}\lowcomma \begin{bsmallmatrix}1\\ 0\end{bsmallmatrix}\bigl)\elowcomma\,\\
 &\bigl(\begin{bsmallmatrix}2\\ 2\end{bsmallmatrix}\lowcomma \begin{bsmallmatrix}0\\ 0\end{bsmallmatrix}\lowcomma \begin{bsmallmatrix}1\\ 1\end{bsmallmatrix}\bigl)\elowcomma\,
 \bigl(\begin{bsmallmatrix}2\\ 2\end{bsmallmatrix}\lowcomma \begin{bsmallmatrix}0\\ 1\end{bsmallmatrix}\lowcomma \begin{bsmallmatrix}1\\ 0\end{bsmallmatrix}\bigl)\elowcomma\,
 \bigl(\begin{bsmallmatrix}2\\ 0\end{bsmallmatrix}\lowcomma \begin{bsmallmatrix}1\\ 1\end{bsmallmatrix}\lowcomma \begin{bsmallmatrix}0\\ 2\end{bsmallmatrix}\bigl)\elowcomma\,
 \bigl(\begin{bsmallmatrix}2\\ 0\end{bsmallmatrix}\lowcomma \begin{bsmallmatrix}1\\ 2\end{bsmallmatrix}\lowcomma \begin{bsmallmatrix}0\\ 1\end{bsmallmatrix}\bigl)\elowcomma\,\\
 &\bigl(\begin{bsmallmatrix}2\\ 1\end{bsmallmatrix}\lowcomma \begin{bsmallmatrix}1\\ 0\end{bsmallmatrix}\lowcomma \begin{bsmallmatrix}0\\ 2\end{bsmallmatrix}\bigl)\elowcomma\,
 \bigl(\begin{bsmallmatrix}2\\ 1\end{bsmallmatrix}\lowcomma \begin{bsmallmatrix}1\\ 2\end{bsmallmatrix}\lowcomma \begin{bsmallmatrix}0\\ 0\end{bsmallmatrix}\bigl)\elowcomma\,
 \bigl(\begin{bsmallmatrix}2\\ 2\end{bsmallmatrix}\lowcomma \begin{bsmallmatrix}1\\ 0\end{bsmallmatrix}\lowcomma \begin{bsmallmatrix}0\\ 1\end{bsmallmatrix}\bigl)\elowcomma\,
 \bigl(\begin{bsmallmatrix}2\\ 2\end{bsmallmatrix}\lowcomma \begin{bsmallmatrix}1\\ 1\end{bsmallmatrix}\lowcomma \begin{bsmallmatrix}0\\ 0\end{bsmallmatrix}\bigl)\bigr\}.
\end{align*}
Order~$\Phi^{\otimes 2}$ as it is written above, and do the  procedure as described in \cref{elimprop}. %
This yields for example a tensor with support
\[
\bigl\{\bigl(\begin{bsmallmatrix}0\\ 0\end{bsmallmatrix}\lowcomma \begin{bsmallmatrix}1\\ 1\end{bsmallmatrix}\lowcomma \begin{bsmallmatrix}2\\ 2\end{bsmallmatrix}\bigr)\elowcomma\,
 \bigl(\begin{bsmallmatrix}1\\ 2\end{bsmallmatrix}\lowcomma \begin{bsmallmatrix}0\\ 0\end{bsmallmatrix}\lowcomma \begin{bsmallmatrix}2\\ 1\end{bsmallmatrix}\bigr)\bigr\}.%
\]
This tensor is isomorphic to $\tens_2$. We thus have $D_{(1,1,1)}^{\otimes 2} \combgeq \tens_2$. Therefore, $\combsubomega(D_{(1,1,1)}) \combgeq \log_2(2)/2 = 0.5$. 
In fact, it is not hard to see that there is a restriction $D_{(1,1,1)} \combgeq \tens_2$, so $\combsubomega(D_{(1,1,1)}) \geq \log_2(2) = 1$.
The final answer is given by Strassen's \cref{strassentripartite}: $\combsubomega(D_{(1,1,1)}) = \subomega(D_{(1,1,1)}) = H(\tfrac13,\tfrac13,\tfrac13) =\log_2(3)$ which is approximately $1.58496$. %
\end{example}

\subsection{Proof of \cref{main}}
We are now ready for the proof of \cref{main}.
In the rest of this section we will use the following notation.

\begin{notation}\label{stdnotation}
Let $\phi$ be a $k$-tensor, $B$ a basis, $\Phi$ the corresponding support. Suppose $\phi$ is tight with respect to $\alpha$.
\begin{itemize}
\item $\optset{P}_{\Phi}$ consists of all probability distributions $P$ on $\Phi$;  and $\marg{P}=(P_1, \ldots, P_k)$ are the marginal distributions of $P$ on the $k$ components respectively;
\item $\optset{R}_{\Phi}$ consists of all subsets $R \subseteq \Phi \times \Phi$ that are not contained in the diagonal set $\Delta_{\Phi}\coloneqq \{(x,x) \mid x\in \Phi\}$, and such that $R\subseteq \mathcal{R}_i$ for some $i\in \{1,\ldots,k\}$, where $\mathcal{R}_i \coloneqq \{(x,y) \in \Phi\times \Phi \mid x_i = y_i\}$.
\item $\optset{Q}_{\Phi, \marg{P},R}$ consists of all probability distributions $Q$ on $R$ whose marginals on the $2k$ components of~$R$ satisfy
$Q_i = Q_{k+i} = P_i$ for $i \in \{1,\ldots, k\}$. %
\end{itemize}
Recall that for $R\subseteq \Phi\times \Phi$, we defined $r(R) = r_\alpha(R)$ to be the rank of the matrix with rows $\{\alpha(x) - \alpha(y) \mid (x,y)\in R\}$ over $\QQ$.
\end{notation}

\begin{customtheorem}{\ref{main}}[Repeated]
The monomial subexponent of $\phi$ is at least,
\[
\combsubomega(\phi) \geq \max_{P\in\optset{P}_{\Phi}}\! \biggl(\!H(P) - (k-2) \max_{\mathclap{R\in \optset{R}_{\Phi}}} \frac{\max_{Q\in \optset{Q}_{\Phi,\marg{P},R}} H(Q) - H(P)}{r_\alpha(R)} \biggr).
\]
\end{customtheorem}

\begin{proof}[\bf\upshape Proof of \cref{main}]
Identify $\Phi$ with $\alpha(\Phi) \subseteq \ZZ^k$.
Let $P = (p_1,p_2,\ldots)$ be a probability distribution on~$\Phi$ with rational probabilities $p_i$. Fix a small number $\varepsilon > 0$ and let $N$ be an integer such that every $Np_i$ is an integer (so $P$ is an $N$-type). Let $\marg{P}=(P_1, \ldots, P_k)$ be the marginal distributions of $P$ on the $k$ components of $\Phi\subseteq \ZZ^k$. 
Let $\Phi^{\otimes N}$ be the support of $\phi^{\otimes N}$ in the tensor product basis. %
 Each element in $\Phi$ corresponds to $k$ integers, so each element in $\Phi^{\otimes N}$ corresponds to $k$ vectors $I_1, I_2, \ldots, I_k \in \ZZ^N$. (See \cref{algill}.)

Let $\psi$ be the restriction of $\phi^{\otimes N}$ obtained by keeping only those elements $(I_1, \ldots, I_k)$ in the support for which the type of $I_i$ is $P_i$ for each $i\in [k]$. These include the elements of type~$P$. By \cref{types}, the number of elements remaining can be bounded as
\begin{equation*}\label{Phiineq}
|\Phi^{\otimes N} \cap (T^N_{P_1} \times \cdots \times T^N_{P_k})| \geq |T_P^N| = \binom{N}{Np_1,Np_2,\ldots} \geq  2^{N H(P) - o(N)},
\end{equation*}
where $T_P^N$ denotes the type class of length $N$ strings of $k$-tuples of integers with type $P$, and $T_{P_i}^N$ denotes the type class of length $N$ strings of integers with type $P_i$.  After the restrictions, the new support is 
\begin{equation*}\label{Psient}
\Psi \coloneqq \Phi^{\otimes N} \cap(T_{P_1}^N \times \cdots \times T_{P_k}^N), \quad |\Psi| \geq 2^{NH(P)-o(N)}.
\end{equation*}

Let $M$ be a prime between $\floor{2^{\mu N}}$ and $2\floor{2^{\mu N}}$ for some $\mu>0$ chosen later (this $M$ exists for any $\mu$ and $N$ by Bertrand's Postulate \cite{aigner2010proofs}). Let $B\subseteq\NN$ be a $(k-1)$-average-free set with
\begin{equation}\label{eqB}
\max(B) < \frac{M}{k-1} \quad\textnormal{and}\quad |B|\geq M^{1-\varepsilon}
\end{equation}
with~$\varepsilon$ as chosen above (such a $B$ exists when $N$ is large enough by \cref{avgfree}). Let $v_1,\ldots, v_N, u_1, \ldots, u_{k-1}$ be independent uniformly random variables in $\ZZ/M\ZZ$, and compute a hash as follows:
\begin{align*}
b_i(I_i) &\coloneqq u_i + \sum_{j=1}^{\smash{N}} (I_i)_j v_j \qquad \textnormal{for $1\leq i\leq k-1$},\\
b_k(I_k) &\coloneqq \frac{1}{k-1}\Bigl( u_1 + \cdots + u_{k-1} - \sum_{j=1}^N (I_k)_j v_j\Bigr),
\end{align*}
with operations understood in $\ZZ/M\ZZ$ (in particular $k-1$ is invertible if $N$ and hence $M$ is large enough).
By construction of $b_1, \ldots, b_k$ and since $\phi$ is tight, 
\begin{equation}\label{eq1}
b_1(I_1) + b_2(I_2) + \cdots + b_{k-1}(I_{k-1}) = (k-1)b_k(I_k)
\end{equation}
holds for every element $(I_1, \ldots, I_k)$ in the support $\Psi$. Using restrictions again, we keep only those elements such that $b_i(I_i) \in B$ modulo $M$ for every $i\in [k]$. Call the remaining tensor $\psi'$ and its support $\Psi'$. Since $\max(B) < \frac{M}{k-1}$, the equality in \eqref{eq1} then holds in $\ZZ$ for the representatives $\{0,1,\ldots,M-1\}$ of $\ZZ/M\ZZ$. By the $(k-1)$-average-free property of $B$ this implies $b_1(I_1) = b_2(I_2) = \cdots = b_k(I_k)$. Summarizing, we restricted $\psi$ to $\psi'$ with a monomial restriction in such a way that for every $(I_1, \ldots, I_k) \in \Psi'=\supp \psi'$ we have $b_1(I_1) = b_2(I_2) = \cdots = b_k(I_k)$.

We apply \cref{elimprop} to $\Psi'$.
This yields a monomial restriction
\begin{equation}\label{combrestr}
\psi' \combgeq \tens_{X-Y}
\end{equation}
where $X = |\Psi'|$ and $Y = |\mathcal{C}'|$ for $\mathcal{C'}\coloneqq\{(I,I') \in \Psi'^2 \mid I\neq I';\, \exists i\, I_i = I'_i\}$. (The elements of $\mathcal{C'}$ are ordered pairs. Equation \eqref{combrestr} also holds if we take unordered pairs, but ordered pairs will be more convenient later.)
We will now find a lower bound on $\E[X-Y]$ (which will also be a lower bound on the largest possible value of $X-Y$) where the expectation is taken over the independent uniform choice of the $u_i$'s and $v_i$'s in $\ZZ/M\ZZ$. We claim to have the following bounds on $\E[X]$ and $\E[Y]$:
\begin{itemize}
\item $\E[X] = |B| |\Psi|  M^{-(k-1)}$ \hfill (Claim 1)
\item $\displaystyle\E[Y] \leq |B|  \max_{R\in \optset{R}_\Phi} \max_{Q\in \optset{Q}_{\Phi,\marg{P},R}} 2^{N H(Q) + o(N)} M^{-(k-1 + r(R))}$ \hfill (Claim 2)
\end{itemize}
where in the second bound $\optset{R}_\Phi$ and $\optset{Q}_{\Phi,\marg{P},R}$ are as in the statement of the theorem. Intuitively, we want to choose $\mu$ small such that $\E[X]$ has a large exponent, but we want to choose $\mu$ large such that $\E[Y]$ has a smaller exponent than the exponent of $\E[X]$.
To make this precise, we put the claimed bounds together with
\begin{itemize}
\item $|B|\geq M^{1-\varepsilon}$ \hfill %
\item $|\Psi| \geq 2^{NH(P) - o(N)}$ \hfill %
\item $M \geq 2^{\mu N - o(N)}$
\end{itemize}
and use linearity of expectation, to obtain
\begin{align*}
\E[X-Y] &= \E[X] - \E[Y]\\
&\geq |B| \Bigl( |\Psi| M^{-(k-1)} - \max_{R\in \optset{R}_\Phi} \max_{Q\in \optset{Q}_{\Phi,\marg{P},R}} 2^{N H(Q) + o(N)} M^{-(k-1+r(R))}  \Bigr)\\
&\geq M^{1-\varepsilon} 2^{N H(P) - o(N)} M^{-(k-1)} \\
&\qquad\qquad \cdot \Bigl( 1 - \max_{R\in \optset{R}_\Phi} \max_{Q\in \optset{Q}_{\Phi,\marg{P},R}} 2^{N(H(Q) - H(P)) + o(N)} M^{-r(R)} \Bigr)\\
&\geq 2^{N(H(P) - (k-2+\varepsilon)\mu) - o(N)}\\
&\qquad\qquad \cdot \Bigl( 1 - \max_{R\in \optset{R}_\Phi} \max_{Q\in \optset{Q}_{\Phi,\marg{P},R}} 2^{N(H(Q) - H(P) - r(R)\mu) + o(N)} \Bigr).
\end{align*}
We want to choose $\mu$ small such that the first factor has a large exponent, but we want to choose $\mu$ large such that the second factor is bounded away from 0. We thus set
\[
\mu = \max_{R\in \optset{R}_\Phi} \frac{\max_{Q\in \optset{Q}_{\Phi,\marg{P},R}} H(Q) - H(P) + o(1)}{r(R)}.
\]
After substituting $\mu$, we get that
\begin{align*}
\combsubomega(\phi) &= \frac{1}{\combomega(\phi, \tens)} = \frac{1}{N \combomega(\phi^{\otimes N}, \tens)} \\
&\geq \frac{1}{N} \log_2 \E[X-Y]\\
&\geq H(P) - (k-2+\varepsilon)\max_{R\in \optset{R}_\Phi} \frac{\max_{Q\in \optset{Q}_{\Phi,\marg{P},R}} H(Q) - H(P)  + o(1)}{r(R)}  - o(1).
\end{align*}
We note again that here we used the probabilistic method: the largest value of $X-Y$ is at least $\E[X-Y]$.
Finally, we let $N \to \infty$ and $\varepsilon \to 0$, and take the supremum over $P$ with rational entries. By continuity, the latter can be replaced by a maximum over arbitrary real-valued probability distributions $P$ on $\Phi$, finishing the main argument of the proof.
It remains to prove the two claims. %

\textbf{Claim 1.} We will prove the claim
\[
\E[X] = |\Psi| |B| M^{-(k-1)}.
\]
Recall that in the main argument we defined $\Psi = \Phi^{\otimes N} \cap (T^N_{P_1} \times \cdots \times T^N_{P_k})$ and $\Psi' = \{(I_1, \ldots, I_k) \in \Psi \mid b_i(I_i) \in B \textnormal{ for all $i\in [k]$}\}$. Recall that for every $(I_1, \ldots, I_k) \in \Psi'$ we have $b_1(I_1) = b_2(I_2) = \cdots = b_k(I_k)$. The random variable~$X$ is the size of $\Psi'$ and can be computed as follows:
\begin{align*}
\E[X] &= \sum_{I\in \Psi} \Pr[ b_1(I_1)\in B \wedge \cdots \wedge b_k(I_k) \in B ] \cdot 1\\
&= \sum_{I\in \Psi}\, \sum_{b\in B}  \Pr[ b_1(I_1) = b_2(I_2) = \cdots = b_k(I_k) = b ]\\
&=  \sum_{I\in \Psi}\, \sum_{b\in B} \Pr[ b_1(I_1) = b_2(I_2) = \cdots = b_{k-1}(I_{k-1}) = b ]\\
&=  \sum_{I\in \Psi} \sum_{b\in B} M^{-(k-1)}\\
&=  |\Psi| |B| M^{-(k-1)},
\end{align*}
using that the random variables $b_1(I_1)$, $b_2(I_2)$, \ldots, $b_{k-1}(I_{k-1})$ are independent uniform in $\ZZ/M\ZZ$ because of the presence of $u_i$ in the definition of $b_i$.

\textbf{Claim 2.} We will prove the claim
\[
\E[Y] \leq |B|  \max_{R\in \optset{R}_\Phi} \max_{Q\in \optset{Q}_{\Phi,\marg{P},R}} 2^{N H(Q) + o(N)} M^{-(k-1 + r(R))}.
\]
The number $Y$ is the size of $\mathcal{C}'=\{(I,I') \in \Psi'^2 \mid I\neq I';\, \exists i\, I_i = I'_i\}$, whose expectation can be written in terms of $\mathcal{C}\coloneqq\{(I,I') \in \Psi^2 \mid I\neq I';\, \exists i\, I_i = I'_i\}$ as %
\begin{align}%
\E[Y] &= \sum_{(I,I')\in \badpairs} \Pr[\forall i\,\, b_i(I_i)\in B \wedge b_i(I'_i) \in B]\nonumber\\
&=  \sum_{\substack{(I,I')\in\badpairs}}\sum_{b\in B} \Pr[b_1(I_1) = \cdots = b_k(I_k)\nonumber\\[-1em] &\hspace{7em}= b_1(I'_1) = \cdots = b_k(I'_k) = b].\label{eqY1}
\end{align}

Fix a pair $((I_1,\ldots,I_k)$, $(I'_1, \ldots, I'_k))\in \badpairs$. %
The random variables $b_i(I_i)$ and $b_i(I'_i)$ are linear combinations of independent uniform random variables, and therefore $(b_1(I_1), \ldots, b_k(I_k), b_1(I'_1), \ldots, b_k(I'_k))$ is uniform on the image subspace in $(\ZZ/M\ZZ)^{2k}$ (see \cref{uniflem}).
This subspace contains $(b, b, \ldots, b)$ for any $b\in \ZZ/M\ZZ$, since $u_1=u_2=\cdots=u_k = b$, $v_1 = \cdots = v_N = 0$ is a possible assignment. The probability of the event $(b,b,\ldots, b)$ is thus equal to the reciprocal of the cardinality of the image subspace. This cardinality equals $M$ to the power the rank of the $(N+k-1)\times 2k$ coefficient matrix
\begin{equation}\label{mat}
\begin{pmatrix}
1 & 0 & \cdots & 0 & \phantom{-}\tfrac{1}{k-1} & 1 & 0 & \cdots & 0 & \phantom{-}\frac{1}{k-1}\\
0 & 1 &  & 0 & \phantom{-}\tfrac{1}{k-1} & 0 & 1 &  & 0 & \phantom{-}\frac{1}{k-1}\\
\vdots &  & \ddots &  & \phantom{-}\vdots  &  & &\ddots  &  & \phantom{-}\vdots\\
0 & 0 & \cdots & 1 & \phantom{-}\tfrac{1}{k-1} & 0 & 0 & \cdots & 1 & \phantom{-}\frac{1}{k-1}\\
I_1 & I_2 & \cdots & I_{k-1} & -\tfrac{I_k}{k-1} & I'_1 & I'_2 & \cdots & I'_{k-1} & -\frac{I'_k}{k-1}
\end{pmatrix}
\end{equation}
over $\ZZ/M\ZZ$, with $I_1, \ldots, I_k, I'_1, \ldots, I'_k$ thought of as column vectors. With column and row operations this matrix can be transformed into
\[
\begin{pmatrix}
0 & 0 & \cdots & 0 & 0 & 1 & 0 & \cdots & 0 & 0\\
0 & 0 & \cdots & 0 & 0 & 0 & 1 &        & 0 & 0\\
\vdots &  \vdots & \ddots & \vdots & \vdots & \vdots  &  & \ddots &  & \vdots\\
0 & 0 & \cdots & 0 & 0 & 0 & 0 &        & 1 & 0\\
I_1 - I'_1 & I_2 - I'_2 & \cdots & I_{k-1} - I'_{k-1} & I_k - I'_k & 0 & 0 & \cdots & 0 & 0
\end{pmatrix}
\]
Here we used that $\sum_{i=1}^k I_i'$ is the zero vector.
It follows that the rank of the matrix in \eqref{mat} is $k-1$ plus the rank of the $N\times k$ matrix
\begin{equation*}\label{eq:mat}
A(I,I')\coloneqq\begin{pmatrix}
I_1 - I'_1 & I_2 - I'_2 & \cdots & I_k - I'_k
\end{pmatrix}.
\end{equation*}
Letting $r_M(I,I')$ denote the rank of $A(I,I')$, equation \eqref{eqY1} thus becomes
\begin{align}\label{eqpr2}
\E[Y] &= \sum_{(I,I')\in \mathcal{C}} \sum_{b\in B} M^{-(k-1+r_M(I,I'))}\nonumber\\
&= |B| \sum_{(I,I')\in \mathcal{C}} M^{-(k-1+r_M(I,I'))}.
\end{align}

Again fix a pair $((I_1,\ldots,I_k)$, $(I'_1, \ldots, I'_k))\in \badpairs$. Recall that $\badpairs \subseteq \Psi \times \Psi \subseteq \Phi^{\otimes N} \times \Phi^{\otimes N}$. Therefore, each row in $A(I,I')$ is of the form $x-y$ for some $(x,y) \in \Phi\times \Phi$. Namely, the $i$th row of $A(I,I')$ equals $x-y$ for $x = ((I_1)_i,\ldots, (I_k)_i)$ and $y = ((I_1')_i,\ldots, (I_k')_i)$.
We map every $(I,I')\in \badpairs$ to the sequence of pairs corresponding to the rows in $A(I,I')$,
\begin{align*}
\mathcal{C} &\to (\Phi\times \Phi)^N\\
(I,I') &\mapsto \bigl(((I_1)_i,\ldots, (I_k)_i),((I_1')_i,\ldots, (I_k')_i)\bigr)_{i\in[N]}.
\end{align*}
The image of this map consists of sequences $S\in (\Phi\times\Phi)^{N}$ such that
\begin{itemize}
\item there is an $R \in \optset{R}_\Phi$
\item there is an $N$-type $Q \in \optset{Q}_{\Phi, \marg{P}, R}$ with $\supp(Q) = R$
\item $S\in T^N_Q$.
\end{itemize}
This map is injective.
In words, $S$ is the sequence of pairs $(x,y)$ corresponding to the rows in the matrix $A(I,I')$, and $R$ is the (unique) set of pairs occurring in $S$, and $Q$ is the (unique) empirical distribution of the elements of $R$ in $S$. Note that the rank of $A(I,I')$ only depends on the set $R$. For any $R\subseteq \Phi\times \Phi$ we thus define $r_M(R)$ to be the rank of the matrix with rows $\{x-y \mid (x,y) \in R\}$ over the field $\ZZ/M\ZZ$ and we write \eqref{eqpr2} as
\[
\E[Y] \leq |B| \sum_{R\in \optset{R}_\Phi}\hspace{0.5em} \sum_{\substack{Q\in \optset{Q}_{\Phi,\marg{P},R}:\\\supp(Q) = R\\ \textnormal{$Q$ is $N$-type}}}\hspace{0.5em} \sum_{S\in T^N_Q} M^{-(k-1 + r_M(R))}.
\]
The number of $N$-types $Q$ on $R$ is at most $\binom{N + |R| - 1}{|R| - 1}$, while the size of the type class $T^N_Q$ is at most $2^{NH(Q)}$ for any $Q\in \optset{Q}_{\Phi, \marg{P}, R}$ (\cref{types}). Therefore, %
\begin{equation}\label{eq2}
\E[Y] %
\leq |B| \sum_{R\in \optset{R}_\Phi} \binom{N + |R| - 1}{|R| - 1} %
\mathop{\vphantom{\sum}\max}_{Q\in \optset{Q}_{\Phi,\marg{P},R}}\hspace{-0.5em} 2^{N H(Q)}\hspace{0.5em} M^{-(k-1 + r_M(R))}.
\end{equation}

Recall that we defined $r(R)$ as the rank of the matrix with rows $x-y$ for $(x,y)\in R$, over~$\QQ$. The value of $r_M(R)$ may be less than~$r(R)$, but only for finitely many primes~$M$. Since there are only finitely many possibilities for choosing $R$, there is an $N_0$ such that for $N > N_0$ and for all $R$ we have $r_M(R) = r(R)$. In the following we assume that $N$ (and thus~$M$) is already large enough for this to hold.
In \eqref{eq2}, we replace $r_M(R)$ by $r(R)$ and upper bound the sum over $R$ by $2^{|\Phi|^2}$ (a constant) times the largest summand, to obtain
\[
\E[Y] \leq |B| 2^{|\Phi|^2} \max_{R\in \optset{R}_\Phi} \binom{N + |R| - 1}{|R| - 1} %
\max_{Q\in \optset{Q}_{\Phi,\marg{P},R}}\hspace{-0.8em} 2^{N H(Q)}\hspace{0.5em} M^{-(k-1 + r(R))}.
\]
Now use that the binomial coefficient is polynomial in $N$.
This completes the proof.
\end{proof}

\subsection{Computational aspects}
We use \cref{stdnotation}.
Before discussing applications of \cref{main}, let us focus on how to compute the bound
\[
\combsubomega(\phi) \geq \max_{P\in\optset{P}_{\Phi}}\! \biggl(\!H(P) - (k-2) \max_{\mathclap{R\in \optset{R}_{\Phi}}} \frac{\max_{Q\in \optset{Q}_{\Phi,\marg{P},R}} H(Q) - H(P)}{r_\alpha(R)} \biggr),
\]
of \cref{main}.

\textbf{Choice of $Q$.}
Given $P$ and~$R$, computing $\max_{Q\in \optset{Q}_{\Phi,\marg{P},R}} H(Q)$ is a convex optimization  problem (and thus easy) since it amounts to maximizing a concave function over a convex set for each $i\in [k]$.

\textbf{Choice of $R$.}
The set $\optset{R}_\Phi$ may be large,
but we can greatly reduce its size as explained in the following two lemmas.

\begin{lemma}\label{subset}
If $R, R' \in \optset{R}_\Phi$ and $R\subseteq R'$ and $r(R) = r(R')$, then 
\[
\frac{\max_{Q\in \optset{Q}_{\Phi, \marg{P},R}} H(Q)}{r(R)} \leq \frac{\max_{Q\in \optset{Q}_{\Phi, \marg{P},R^{\smash{\prime}}}} H(Q)}{r(R')}.
\]
\end{lemma}
\begin{proof}
We leave the proof to the reader.
\end{proof}

Recall that a subset $R\subseteq \Phi \times \Phi$ is an \emph{equivalence relation} if for all $x,y,z\in \Phi$ we have $(x,x) \in R$; $(x,y) \in R \Rightarrow (y,x) \in R$; and $(x,y)\in R \wedge (y,z)\in R \Rightarrow (x,z) \in R$. If $\Phi = \sqcup_i \Phi_i$ is a set partition of $\Phi$, then $R = \{(x,y)\in \Phi\times\Phi \mid \exists i: x\in \Phi_i \wedge y\in \Phi_i\}$ is an equivalence relation. The set partition into singletons yields the \emph{equality} equivalence relation $\Delta_\Phi = \{(x,x) \mid x\in \Phi\}$.

\begin{lemma}\label{equivrel}
Let $R\in \optset{R}_\Phi$. Then there is an $R'\in \optset{R}_\Phi$ which is an equivalence relation and such that $R\subseteq R'$ and $r(R) = r(R')$.
Therefore,  \cref{main} still holds if $\optset{R}_\Phi$ is replaced by the set
\[
\optset{R}_\Phi \cap \{\textnormal{equivalence relations}\}.
\]
\end{lemma}
\begin{proof}
We can extend $R$ to an equivalence relation without increasing its rank. Namely, let $\Delta_\Phi = \{(x,x) \mid x\in \Phi\}$,  $R^T = \{(y,x) \mid (x,y) \in R\}$ and $R\circ R = \{(x,z) \mid (x,y), (y,z) \in R\}$. Then $r(R\cup \Delta_\Phi)=r(R)$, $r(R\cup R^T)=r(R)$ and  $r(R \cup (R\circ R))=r(R)$.  %
As mentioned above, if $R\subseteq R'$ with $r(R) = r(R')$, then only $R'$ needs to be considered. Therefore, replacing $\optset{R}_\Phi$ by $\optset{R}_\Phi\cap \{\textnormal{equivalence relations}\}$ will not change the optimal value of the maximization over $\optset{R}_\Phi$.
\end{proof}

We state one more simple fact.

\begin{lemma}\label{kmintwo}
$r(\mathcal{R}_i) \leq k-2$.
\end{lemma}
\begin{proof}
The support $\Phi$ is tight with respect to $\alpha$, so for every $x\in \Phi$ we have $\sum_{j=1}^{k} \alpha_j(x_j) = 0$. So any $w\in \Span_\QQ\{\alpha(x) - \alpha(y) \mid (x,y) \in \mathcal{R}_i\}$ satisfies the equations $w_i = 0$ and $\sum_{j=1}^{k} w_j = 0$ and hence the span has rank at most~$k-2$.
\end{proof}

\textbf{Choice of $\alpha$.} Finally, we discuss the choice of the map $\alpha = (\alpha_1, \ldots, \alpha_k)$. If $\alpha$ and $\beta$ are such maps then $\phi$ is also tight for $\gamma = (\gamma_1, \ldots, \gamma_k)$ with $\gamma_i\coloneqq\alpha_i+C\beta_i$ for any integer $C$, except that $\gamma_i$ may fail to be injective for finitely many $C$. For a given $R\in \optset{R}_\Phi$, $r_\alpha(R)$ is the rank of some matrix $X$ and $r_\beta(R)$ is the rank of some other matrix $Y$, while $r_\gamma(R)$ is the rank of $X+CY$. But the latter is at least $\max\{r_\alpha(R),r_\beta(R)\}$ except for at most finitely many values of $C$. Taking into account every $R$ still excludes only finitely many values of $C$, hence there is at least one good $C$. So in theory one can proceed by finding a new $k$-tuple of injective maps with at least one $r(R)$ higher than what one already has and then improve by finding suitable linear combination. There are finitely many relations and the ranks are never larger than $k-2$, therefore only finitely many improvements are possible. %

\subsection{Applications of \cref{main}}\label{appl}
We finish this section by giving some example applications of \cref{main}, one of which is the result on the weight-$(2,2)$ Dicke tensor that we use in combination with \cref{cwcomplete} to upper bound the exponent of the complete graph tensor. We use \cref{stdnotation}.

First, as a sanity check, we derive from \cref{main} the lower-bound part of Strassen's \cref{strassentripartite} on the monomial subexponent of tight 3-tensors.

\begin{lemma}\label{chain}
Fix marginal distributions $\marg{P} = (P_1, \ldots, P_k)$ on the $k$ components of $\Phi$ and let $P\in \optset{P}_\Phi$ be an element that has maximal entropy among all elements of $\Phi$ with marginals $P_1, \ldots, P_k$.
Let $R\in \optset{R}_\Phi$ such that $R\subseteq \mathcal{R}_i$ and let $Q \in \optset{Q}_{\Phi,\marg{P},R}$. Then,
\[
H(Q) \leq 2H(P) - H(P_i).
\]
\end{lemma}

\begin{proof}
Let $\overline{i}=\{1, \ldots, \widehat{i}, \ldots, k\}$ and let $k+\overline{i}=\{k+1, \ldots, \widehat{k+i}, \ldots, 2k\}$, where a hat denotes an omitted index. In this proof, if $K\subseteq[2k]$ is a set of indices, then for any $Q\in \optset{Q}_{\Phi,\marg{P},R}$ we will write $Q_K$ for the marginal distribution of $Q$ on these indices.
Let ($\ast$) denote the assumption that among all distributions in $\optset{P}_\Phi$ with marginals $P_1, \ldots, P_k$ the distribution $P$ is the one with maximal entropy. %
Let $Q\in \optset{Q}_{\Phi,\marg{P},R}$.
Then, 
\begin{align*}
H(Q) %
&= H(Q_i) + H(Q_{\overline{i} \cup (k+[k])} \mid Q_i) \tag{entropy chain rule}\\
     &\leq H(Q_i) + H(Q_{\overline{i}} \mid Q_i) + H(Q_{k + [k]} \mid Q_i) \tag{strong sub-additivity}\\
     &\leq H(Q_i) + H(Q_{\overline{i}\cup\{i\}}) - H(Q_i)\\
     &\phantom{{}\leq H(Q_i){}} {}+ H(Q_{(k+[k]) \cup \{i\}}) - H(Q_i) \tag{entropy chain rule}\\
     &\leq H(Q_i) + H(Q_{[k]}) - H(Q_i)\\
     &\phantom{{}\leq H(Q_i){}} {}+ H(Q_{k+[k]}) - H(Q_i) \tag{since $R\subseteq \mathcal{R}_i$}\\
     &\leq 2H(P) - H(P_i), \tag{by ($\ast$)}
\end{align*}
which proves the lemma.
\end{proof}

\begin{corollary}[Strassen \cite{strassen1991degeneration}]
Let $\phi$ be a tight 3-tensor. Then
\[
\combsubomega(\phi) \geq \max_{P\in \optset{P}_\Phi} \min \{ H(P_1), H(P_2), H(P_3) \}.
\]
\end{corollary}
\begin{proof}
Let $P\in\optset{P}_\Phi$. Then by \cref{main},
\begin{equation}\label{streq}
\combsubomega(\phi) \geq H(P) - \max_{R\in \optset{R}_\Phi}  \frac{\max_{Q\in \optset{Q}_{\Phi,\marg{P},R}} H(Q) - H(P)}{r_\alpha(R)}.
\end{equation}
Let $P_1, \ldots, P_k$ be the marginals of $P$. We may assume that $P$ has maximal entropy among all elements of $\optset{P}_\Phi$ with marginals $P_1, \ldots, P_k$, since $\optset{Q}_{\Phi,\marg{P},R}$ depends on $P$ only through its marginals.
Let $R\in \optset{R}_\Phi$. Then $r(R) \geq 1$ (since $R\not\subseteq \Delta_{\Phi}$) and $r(R) \leq k-2 = 1$ (\cref{kmintwo}). Therefore, $r(R) = 1$. %
Combine this fact with \eqref{streq} and \cref{chain} to obtain
\[
\combsubomega(\phi) \geq H(P) - \max_{i} (2H(P) - H(P_i)  - H(P)) = \min_{i} H(P_i),
\]
which proves the corollary.
\end{proof}

Second, we derive from \cref{main} the lower-bound part of \cref{wstate} on the subexponent of the W-state tensors $W_k = D_{(1, k-1)}$. %

\begin{corollary}[Vrana--Christandl \cite{vrana2015asymptotic}]\label{wstate2} Let $k\geq 3$. %
Then
\[
\combsubomega(W_k) = \subomega(W_k) = h\bigl(k^{-1} \bigr),
\]
where $h(p)$ denotes the binary entropy function.
\end{corollary}

\begin{proof}%
For the proof of the upper bound $\subomega(W_k) \leq h(k^{-1})$ we refer to the proof of Theorem~11 in \cite{vrana2015asymptotic}.
We will here give a proof for the lower bound $\combsubomega(W_k) \geq h(k^{-1})$, fixing a small gap in the proof of \cite{vrana2015asymptotic}. Let $\Phi$ be $\supp W_k$ which we identify with the set $\{(1,0,\ldots,0), (0,1,\ldots, 0), \ldots, (0,0,\ldots, 1)\}$.
Let~$P$ be the uniform probability distribution on $\Phi$. Then the marginals of~$P$ are $P_i(1) = 1 - P_i(0) = \tfrac1k$.  By \cref{main},
\begin{equation}\label{vcmain}
\combsubomega(W_k) \geq \log_2 k - (k-2) \max_{R\in \optset{R}_\Phi}\frac{\max_{Q\in \optset{Q}_{\Phi,\marg{P},R}} H(Q) - \log_2 k}{r(R)}.
\end{equation}

We may assume that $\optset{R}_\Phi$ consists of equivalence relations $R\subseteq \mathcal{R}_1$ by \cref{equivrel} and the symmetry of $\Phi$. An equivalence relation on $\Phi$ is just a set partition of $\Phi$ and we say that the \emph{type} of such a set partition is the integer partition consisting of the sizes of the parts occurring in the set partition.  Up to permuting the elements of $\Phi$, an equivalence relation on $\Phi$ is characterized by its type. %
Let $\lambda = (\lambda_1, \lambda_2, \ldots, \lambda_{\ell(\lambda)}) \vdash k$ be an integer partition with $\ell(\lambda)$ parts. Let $R_\lambda \in \optset{R}_\Phi$ be any equivalence relation of type $\lambda$. Since $R_{\lambda} \subseteq \mathcal{R}_1$, we have $\lambda_{\ell(\lambda)} = 1$.
It is not difficult to see that $r(R_\lambda) = k - \ell(\lambda)$ (with $\alpha$ chosen as in \cref{Dtight}); indeed if $C_1, \ldots, C_{\ell(\lambda)}$ are the equivalence classes  of $R_\lambda$ in nonincreasing order, then $r(C_i) = \lambda_i - 1$ and $r(R_\lambda) = \sum_{i=1}^{\ell(\lambda)} r(C_i)$ since $\phi = W_k$. Therefore,
\begin{equation}\label{vcrank}
r(R_\lambda) = \sum_{i=1}^{\smash{\ell(\lambda)}} (\lambda_i - 1) = k - \ell(\lambda).
\end{equation} 
A probability distribution $Q$ on $R_\lambda$ can be identified with a block-diagonal matrix with $\lambda_i\times \lambda_i$ blocks, and the condition on the marginals means that the row and column sums are $\tfrac{1}{k}$. Using permutation symmetry and concavity of the entropy and permutation symmetry of $\optset{Q}_{\Phi, \marg{P},R}$, one can see that $\max_{Q\in \optset{Q}_{\Phi,\marg{P},R}} H(Q)$ is attained when~$Q$ is constant in each block.  In this case
\begin{equation}\label{vcQ}
H(Q) = \sum_{i} \lambda_i^2\, \Bigl( \frac{1}{\lambda_i k} \log_2 (\lambda_i k) \Bigr) = 2\log_2 k - H\Bigl(\frac{\lambda}{k}\Bigr).
\end{equation}
Combine \eqref{vcmain}, \eqref{vcrank} and \eqref{vcQ} to get
\[
\combsubomega(W_k) \geq  \log_2 k - (k-2) \max_{\lambda} \frac{\log_2 k - H(\tfrac{\lambda}{k})}{k - \ell(\lambda)}.
\]

If $k$ and $\ell = \ell(\lambda)$ are fixed, then the minimum of $H(\tfrac{\lambda}{k})$ is attained at the partition $\lambda = (k-\ell + 1, 1^{\ell - 1})$, because the other distributions can be expressed as a convex combinations of its permutations. Therefore we need to look for the maximum of
\[
\frac{\log_2 k - H\bigl(\tfrac{(k - \ell + 1,\, 1^{\ell - 1})}{k}\bigr)}{k-\ell} = \frac{(k-\ell + 1) \log_2(k - \ell + 1)}{k(k - \ell)}.
\]
Taking the derivative with respect to $\ell$, we get
\[
\frac{d}{d\ell} \frac{(k-\ell + 1) \log_2(k - \ell + 1)}{k(k - \ell)} = \frac{\ln(1 + k - \ell) - (k - \ell)}{k(k-\ell)^2\ln2} \leq 0,
\]
using $\ln(1+x) \leq x$. Therefore, the optimum is at $\ell = 2$ and this gives
\[
\combsubomega(W_k) \geq \log_2 k - \bigl( \log_2 k - h\bigl( \tfrac{1}{k}\bigr) \bigr) = h(k^{-1}),
\]
which proves the corollary.
\end{proof}

Finally, we compute the (monomial) subexponent of the weight-$(2,2)$ Dicke tensor.

\begin{corollary}[Weight-$(2,2)$ Dicke tensor] Let $D_{(2,2)}$ be the weight-$(2,2)$ Dicke tensor, 
\begin{alignat*}{4}
D_{(2,2)} ={}&\,  b_0\otimes b_0\otimes b_1\otimes b_1
 {}+{}&\, b_0\otimes b_1\otimes b_0\otimes b_1
 {}+{}&\, b_0\otimes b_1\otimes b_1\otimes b_0\\
 {}+{}&\, b_1\otimes b_0\otimes b_0\otimes b_1
 {}+{}&\, b_1\otimes b_0\otimes b_1\otimes b_0
 {}+{}&\, b_1\otimes b_1\otimes b_0\otimes b_0.
\end{alignat*}
Then $\combsubomega(D_{(2,2)}) = \subomega(D_{(2,2)}) = 1$.
\end{corollary}
\begin{proof}
Let $P$ be the uniform probability distribution on the support $\Phi$ of~$D_{(2,2)}$ which we identify with $\{(1,1,0,0), (1,0,1,0), \ldots\}$. Then $H(P) = \log_2 6$.
Note that the marginals $P_i$ are uniform distributions on $\{0,1\}$.
Let $R\subseteq \optset{R}_\Phi$. By the permutation symmetry of $D_{(2,2)}$, we may assume that 
\begin{align*}
R\subseteq \mathcal{R}_1 = \phantom{\cup\,}&\{(1,1,0,0), (1,0,1,0), (1,0,0,1)\}^2 \\
\cup\, &\{(0,0,1,1), (0,1,0,1), (0,1,1,0)\}^2.
\end{align*}
By \cref{equivrel}, we may assume that $R$ is an equivalence relation. Define the set
\[
S = \{(1,1,0,0), (1,0,1,0), (1,0,0,1)\}.
\]
If $(x,y) \in R$, then $R \subseteq R'\coloneqq R\cup\{((1,1,1,1)-x, (1,1,1,1)-y)\}\in \optset{R}_\Phi$ and $r(R) = r(R')$, and hence we may assume that if $(x,y)\in R$ then also $((1,1,1,1)-x, (1,1,1,1)-y) \in R$ (\cref{subset}).  We thus restrict ourselves to the equivalence relations $R\subseteq \mathcal{R}_1$ of the following three types. In type $(3)$ all three elements of $S$ are equivalent. There is only one such equivalence relation, namely the whole set $\mathcal{R}_1$, which has size 18. In type $(2,1)$, two elements of~$S$ are mutually equivalent and inequivalent to the third element (which is equivalent to itself). Then the size of $R$ is 10. In type $(1,1,1)$, all elements of $S$ are inequivalent. However, this means that $R$ is contained in the diagonal $\Delta_\Phi = \{(x,x) \mid x\in \Phi\}$. Such an $R$ is not feasible. So we are left with the two types $(3)$ and $(2,1)$. For type $(3)$ we have $r(R) = 2$ while for type $(2,1)$ we have $r(R) = 1$. In both cases, the uniform probability distribution~$Q$ on $R$ has marginals $P_1, P_2, P_3, P_4, P_1, P_2, P_3, P_4$, namely uniform distributions on $\{0,1\}$. Therefore, this $Q$ is the optimal~$Q$ and $H(Q) = \log_2 |R|$. Letting $R_{(3)}$ and $R_{(2,1)}$ be any equivalence relations of type $(3)$ and $(2,1)$ respectively, we obtain
\begin{align*}
\combsubomega(D_{(2,2)}) &\geq \min \bigl\{ H(P) - \tfrac{2}{r(R_{(3)})}\bigl(\log_2 |R_{(3)}| - H(P)\bigr),\\[-0.2em]
& \phantom{{}\geq \min \bigl\{{}}H(P) - \tfrac{2}{r(R_{(2,1)})}\bigl(\log_2 |R_{(2,1)}| - H(P)\bigr)\bigr\}\\
&= \min\{{} \log_2 6 - \tfrac{2}{2}(\log_2 18 - \log_2 6),\\
&\phantom{{}= \min \{{}}\log_2 6 - \tfrac{2}{1}(\log_2 10 - \log_2 6)\} \\
&= \min\{1, \log_2\tfrac{54}{25}\} = 1
\end{align*}
On the other hand, $D_{(2,2)}$ has a flattening of rank 2, so also the upper bound $\subomega(D_{(2,2)}) \leq 1$ holds.
\end{proof}

\paragraph{Acknowledgements.} The authors wish to thank Peter Bürgisser for helpful discussions. MC acknowledges financial support from the European Research Council (ERC Grant Agreement no.~337603), the Danish Council for Independent Research (Sapere Aude), and VILLUM FONDEN via the QMATH Centre of Excellence (Grant no.~10059). JZ is supported by NWO through the research programme 617.023.116.

\raggedright
\bibliographystyle{alphaurlpp}
\bibliography{completegraphanddistillation}

\end{document}